\documentclass[11pt]{amsart}
\usepackage{hyperref}
\usepackage{amsfonts}
\usepackage{amsmath}
\usepackage{amssymb}
\usepackage{amscd}
\usepackage{graphicx}
\usepackage{enumitem}
\usepackage{latexsym}
\usepackage{color}
\usepackage{epsfig}
\usepackage{amsopn}
\usepackage{mathrsfs}
\usepackage{array}
\usepackage{arydshln}
\usepackage{fancyvrb}

\usepackage{setspace}

\theoremstyle{plain}
\newtheorem{lemma}{Lemma}[section]
\newtheorem*{theorem*}{Theorem}
\newtheorem*{lemma*}{Lemma}
\newtheorem*{proposition*}{Proposition}
\newtheorem*{conjecture*}{Conjecture}
\newtheorem*{corollary*}{Corollary}
\newtheorem*{problem*}{Problem}
\newtheorem{theorem}[lemma]{Theorem}
\newtheorem{conjecture}[lemma]{Conjecture}
\newtheorem{corollary}[lemma]{Corollary}
\newtheorem{proposition}[lemma]{Proposition}

\theoremstyle{definition}
\newtheorem{definition}[lemma]{Definition}
\newtheorem{example}[lemma]{Example}
\newtheorem{remark}[lemma]{Remark}

\let\tilde\widetilde

\newcommand{\Z}{\mathbb{Z}}

\newcommand{\C}{\mathbb{C}}
\newcommand{\Q}{\mathbb{Q}}

\newcommand{\OO}{\mathcal{O}}
\newcommand{\te}{\otimes}

\newcommand{\cW}{\mathcal W}
\newcommand{\cK}{\mathcal K}

\newcommand{\frakm}{\mathfrak m}

\renewcommand{\P}{\mathbb{P}}
\newcommand{\cL}{\mathcal{L}}

\DeclareMathOperator{\Hom}{Hom}

\DeclareMathOperator{\Cof}{Cof}
\DeclareMathOperator{\Sym}{Sym}

\DeclareMathOperator{\cond}{cond}

\DeclareMathOperator{\Ext}{Ext}

\DeclareMathOperator{\GL}{GL}
\DeclareMathOperator{\PGL}{PGL}
\DeclareMathOperator{\PSL}{PSL}
\DeclareMathOperator{\SL}{SL}

\DeclareMathOperator{\spn}{span}

\DeclareMathOperator{\edim}{edim}

\DeclareMathOperator{\reg}{reg}

\begin{document}

\author[Th. Bauer]{Thomas Bauer}
\address{Fach\-be\-reich Ma\-the\-ma\-tik und In\-for\-ma\-tik,
   Philipps-Uni\-ver\-si\-t\"at Mar\-burg,
   Hans-Meer\-wein-Stra{\ss}e,
   D-35032~Mar\-burg, Germany.}
\email{tbauer@mathematik.uni-marburg.de}
\author[S. Di Rocco]{Sandra Di Rocco}
\address{Department of Mathematics, KTH, 100 44 Stockholm, Sweden.}
\email{dirocco@math.kth.se}
\author[B. Harbourne]{Brian Harbourne}
\address{Department of Mathematics, University of Nebraska-Lincoln, Lincoln, NE, 68588, USA}
\email{bharbourne1@unl.edu}
\author[J. Huizenga]{Jack Huizenga}
\address{Department of Mathematics, The Pennsylvania State University, University Park, PA 16802}
\email{huizenga@psu.edu}
\author[A. Seceleanu]{Alexandra Seceleanu}
\address{Department of Mathematics, University of Nebraska-Lincoln, Lincoln, NE, 68588, USA}
\email{aseceleanu@unl.edu}
\author[T. Szemberg]{Tomasz Szemberg}
\address{Instytut Matematyki UP, Podchor\c a\.zych 2, PL-30-084 Krak\'ow, Poland.}
\email{tomasz.szemberg@gmail.com}
\thanks{During the preparation of this manuscript, Th.~Bauer was partially supported by DFG grant BA~1559/6--1, S. Di Rocco was partially supported by the VR grants [NT:2010-5563, NT:2014-4763],
B. Harbourne was partially supported by NSA grant H98230-13-1-0213, J. Huizenga was partially supported by NSF grant DMS-1204066 and NSA grant H98230-16-1-0306, A. Seceleanu was partially supported by NSF grant DMS-1601024 and T. Szemberg was partially supported by National Science Centre, Poland, grant 2014/15/B/ST1/02197.}

\title[Negative curves on symmetric blowups]{Negative curves on symmetric blowups of the projective plane, resurgences and Waldschmidt constants}

\begin{abstract}
The Klein and Wiman configurations are highly symmetric configurations of lines in the projective plane arising from complex reflection groups.  One noteworthy property of these configurations is that all the singularities of the configuration have multiplicity at least three.  In this paper we study the surface $X$ obtained by blowing up $\P^2$ in the singular points of one of these line configurations.  We study invariant curves on $X$ in detail, with a particular emphasis on curves of negative self-intersection.  We use the representation theory of the stabilizers of the singular points to discover several invariant curves of negative self-intersection on $X$, and use these curves to study Nagata-type questions for linear series on $X$.

The homogeneous ideal $I$ of the collection of points in the configuration is an example of an ideal where the symbolic cube of the ideal is not contained in the  square of the ideal; ideals with this property are seemingly quite rare.  The \emph{resurgence} and \emph{asymptotic resurgence} are invariants which were introduced to measure such failures of containment. We use our knowledge of negative curves on $X$ to compute the resurgence of $I$ exactly.  We also compute the asymptotic resurgence and Waldschmidt constant exactly in the case of the Wiman configuration of lines, and provide estimates on both for the Klein configuration.
\end{abstract}

\date{\today}



\maketitle

\setcounter{tocdepth}{1}
\tableofcontents

\section{Introduction}

In recent years configurations of points in $\P^2$ arising as the singular loci of line configurations have provided examples of many interesting phenomena in commutative algebra and birational geometry.  The dual Hesse conifiguration of 12 points and, more generally, the Fermat configurations of $n^2+3$ points studied in \cite{DST13,HS14,NS16} arise as singular points of Ceva line arrangements which correspond to the reflection groups $G(n,n,3)$.  In this paper we focus instead on the sporadic Klein and Wiman point configurations of 49 and 201 points.  These are the singular points of line arrangements $\cK$ and $\cW$  arising from reflection groups $\PSL(2,7)$ and $A_6$.  We give a detailed study of the surfaces $X_\cK$ and $X_\cW$ obtained by blowing up the points in the configuration, with the particular goal of studying curves of negative self-intersection.

The Klein and Wiman line configurations arise naturally from subgroups $G\subset \PGL_3(\C)$ of automorphisms of $\P^2$.  In the case of the Klein configuration, we denote $G$ by $G_\cK$; it is isomorphic to $\PSL(2,7)$, the finite simple group of order $168$ which is the automorphism group of the Klein quartic curve $$x^3y+y^3z+z^3y = 0.$$  This group has $21$ involutions, each of which fixes a line in $\P^2$; the Klein configuration $\cK$ consists of these $21$ lines.  They meet in $21$ quadruple points and $28$ triple points, and have no further singularities.  The group $G_\cK$ acts transitively on the lines, on the quadruple points, and on the triple points.
Similarly, the Wiman configuration $\cW$ consists of $45$ lines meeting in $36$ quintuple points, $45$ quadruple points, and $120$ triple points, and arises from a subgroup $G_\cW\subset \PGL_3(\C)$ isomorphic to the alternating group $A_6$.
See Section \ref{sec-prelim} for additional background on the Klein and Wiman configurations.

\subsection{Waldschmidt constants and a Nagata-type theorem} For a line configuration $\cL$ in $\P^2$ we let $I_{\cL}\subset S := \C[x,y,z]$ denote the homogeneous ideal of the collection of singular points in the line configuration.  If $I\subset S$ is the ideal of a reduced collection of distinct
points $p_1,\ldots,p_n\in \P^2$, then we define the $m$th symbolic power $I^{(m)}=\bigcap_i I_{p_i}^m$, where $I_{p_i}$ is the homogeneous ideal of the point $p_i$.  That is, $I^{(m)}$ is the ideal
generated by all homogeneous forms vanishing to order at least $m$ at each of the points $p_i$.  The \emph{Waldschmidt constant} $\widehat\alpha(I)$
\cite{W77, D13, DHST14} is defined to be the limit
$$\widehat\alpha (I) = \lim_{m\to \infty} \frac{\alpha(I^{(m)})}{m},$$
where $\alpha(J)$, for a nonzero ideal $J$, denotes the minimal degree among nonzero elements of $J$
(see also \cite{Ch81, EV83}). It is always true that $1\leq\widehat\alpha (I)\leq \sqrt{n}$; for $n\geq 10$ sufficiently general points $p_i$, the famous conjecture of Nagata asserts that $\widehat\alpha(I) = \sqrt{n}$ \cite{N59, CHMR13}.  On the other hand, for special collections of $n$ points, the Waldschmidt constant is typically smaller than $\sqrt{n}$.
Our first main theorem, Theorem \ref{thm-WaldschmidtIntro}, gives our best result on the values of the Waldschmidt constants of the ideals $I_\cK$ and $I_\cW$ and
provides an example of this.


\begin{theorem}\label{thm-WaldschmidtIntro}
For the Klein configuration $\cK$ of $21$ lines, we have $$6.480\approx \frac{661}{102}\leq\widehat \alpha(I_\cK) \leq  6.5.$$ For the Wiman configuration $\cW$ of $45$ lines, we have $$\widehat\alpha(I_\cW) = \frac{27}{2}.$$
\end{theorem}

In each case it is fairly easy to bound the Waldschmidt constant $\widehat\alpha(I_\cL)$ from above by constructing curves with appropriate multiplicities.  These upper bounds rely only on the incidence properties of the line configuration, and in particular make minimal use of the group $G$ of symmetries.  On the other hand, we will see that lower bounds on the Waldschmidt constant of $I_\cL$ can be obtained by proving that certain $G$-invariant divisor classes $D$ on the blowup $X_\cL$ are nef.  Our proof that such divisors are actually nef will rely heavily on the group action.

\subsection{Invariant linear series}

Suppose $D$ is an effective $G$-invariant divisor class on $X_\cL$, and that we would like to prove $D$ is nef.  If $D$ were not nef, then the base locus of the complete series $|D|$ would contain a curve of negative self-intersection.  Since $D$ is $G$-invariant, the base locus of $|D|$ is additionally $G$-invariant.  Therefore there is a $G$-invariant curve of negative self-intersection on $X_\cL$ which meets $D$ negatively.

This observation suggests that we should study linear series of invariant curves on $X_\cL$ in greater detail.  For simplicity, let us discuss the case of the Klein configuration $\cK$.  Suppose $C$ is a $G=G_\cK$-invariant curve on $X_\cK$ which does not contain the line configuration.  Then we will see that the defining equation of $C$ is a polynomial in some fundamental invariant forms $\Phi_4,\Phi_6,\Phi_{14}$ of degrees $4,6,14$, respectively.  Letting $T = \C[\Phi_4,\Phi_6,\Phi_{14}]\subset S$, we define a vector space $$T_d(-m_4E_4-m_3E_3)\subset T_d$$ consisting of degree $d$ forms which are $m_4$-uple at the quadruple points in the configuration and $m_3$-uple at the triple points in the configuration.  Elements of this vector space define $G$-invariant curves in the linear series $|dH - m_4E_4 - m_3E_3|$ on $X_\cK$, where we write $H$ for the class of a line and $E_m$ for the sum of the exceptional divisors over the $m$-uple points in the configuration.

It is not immediately obvious what we should expect the dimension of the linear series $T_d(-m_4E_4-m_3E_3)$ to be.  For instance, we will see that any invariant curve passing through one of the triple points of the configuration is actually double there, so that the obvious conditions cutting $T_d(-m_4E_4-m_3E_3)$ out as a subspace of $T_d$ are typically non-independent.  Our key insight is to study the action of the stabilizer $G_p$ of $p$ on the local ring $(\OO_p,\frakm_p)$ at a point $p$ of the configuration.  If $C$ is a $G$-invariant curve which has multiplicity $k$ at $p$ then the tangent cone of $C$ at $p$ must be $G_p$-invariant. If $f\in \frakm_p^k/\frakm_p^{k+1}$ defines the tangent cone then $G_p$ acts by a linear character on $f$, but in our situation this character is trivial and $f$ is $G_p$-invariant.  Therefore in any vector space $V\subset T_d$ of forms that have a $k$-uple point at $p$, the codimension of the subspace of forms with a $(k+1)$-uple point at $p$ is at most $\dim (\frakm_p^k/\frakm_p^{k+1})^{G_p}$.  The stabilizers $G_p$ are small dihedral groups and these dimensions are easy to compute, which leads to the following theorem.

\begin{theorem}\label{thm-expDimIntro}
Define the \emph{expected dimension} of the vector space $T_d(-m_4E_4-m_3E_3)$ to be $$\edim T_d(-m_4 E_4-m_3E_3)= \max\{\dim T_d - \cond_4(m_4) - \cond_3(m_3),0\},$$ where $\cond_n(m)$ is the number of monomials of degree less than $m$ in a polynomial algebra $\C[u,v]$ where $\deg u = 2$ and $\deg v = n$.  Then we have
$$\dim T_d(-m_4E_4-m_3E_3)  \geq \edim T_d (-m_4E_4-m_3E_3).$$
\end{theorem}

This notion of expected dimension is useful because it appears to be a reasonably good approximation to the dimension.  In Section \ref{sec-invariantSeries} we make an SHGH-type conjecture which in particular implies that the actual and expected dimension coincide unless there is an obvious geometric reason for them not to; the conjecture has been verified by computer so long as $d < 144$ (see \cite{Se61, Ha86, Gi87, Hir89} for the original SHGH Conjecture,
and \cite{CHMR13} for exposition).

\subsection{Explicit curves of negative self-intersection}
Our  results on invariant linear series allow us to study explicit negative curves on $X_\cL$ in detail.  When $G$ is a group acting on a surface, we say that a $G$-invariant curve is \emph{$G$-irreducible} if it has a single orbit of irreducible components.  For example, since $G_\cL$ acts transitively on the lines in $\cL = \cK$ or $\cW$, the sum of the lines in $\cL$ is $G_\cL$-irreducible.
\begin{theorem}\label{thm-curvesIntro}
There is a unique curve of class $42H - 8E_3$ on $X_\cK$.  It is $G_\cK$-invariant, $G_\cK$-irreducible, and reduced.

There is a unique curve of class $90H - 4E_4 - 8E_3$ on $X_\cW$.  It is $G_\cW$-invariant, $G_\cW$-irreducible, and reduced.
\end{theorem}

We use these curves to prove that certain key divisors $D$ are nef, and lower bounds on the Waldschmidt constant $\widehat\alpha(I_\cL)$ follow.  In the case of the Wiman configuration, this lower bound matches the easy upper bound, and we compute $\widehat \alpha (I_{\cW}) = \frac{27}{2}$ exactly.  The computations proving Theorem \ref{thm-curvesIntro} form the technical core of the paper.

Note that the divisor class $42H-8E_3$ on $X_\cK$ is effective by Theorem \ref{thm-expDimIntro}, since the expected dimension of $T_{42}(-8E_3)$ is $1$.  Verifying that there is a $G_\cK$-irreducible curve of this class still takes considerable additional effort, however.

On the other hand, the class $90H-4E_4-8E_3$ on $X_\cW$ is not obviously effective, as the expected dimension of $T_{90}(-4E_4-8E_3)$ is $0$.  The existence of this curve is quite surprising, as the ``local'' conditions to have the given multiplicities at the different points fail to be globally independent.  Some amount of computation seems unavoidable, but the representation-theoretic results of Section \ref{sec-invariantSeries} streamline things considerably.

\subsection{Resurgence, asymptotic resurgence, and failure of containment}

Let $I\subset S = \C[x,y,z]$ be the homogeneous ideal of a finite set of points in $\P^2$.  It follows from either Ein-Lazarsfeld-Smith \cite{ELS01} or Hochster-Huneke \cite{HoHu02} that $I^{(4)}\subseteq I^2$.  On the other hand, Huneke asked whether $I^{(3)}\subseteq I^2$ is also true (see \cite{PSC, HaHu13} for discussion and generalizations).  It is now known that $I^{(3)}\subseteq I^2$ can fail \cite{DST13, BCH14, HS14, CGetal16, DHNSST15}
(see also \cite{SzeSzp17} for a compact and up to date overview), but failures seem quite rare and it is an open problem to characterize which configurations of points exhibit this failure of containment. Whether other similar failures, such as $I^{(5)}\not\subseteq I^3$ or more generally
$I^{(2r-1)}\not\subseteq I^r$ for $r>2$, ever occur over $\C$ remains open \cite{PSC, HaHu13} (but see also \cite{HS14}).

The containment $I^{(3)}\subseteq I^2$ typically holds even for ideals of the form $I=I_\cL$ (see, for example, \cite[Example 8.4.8]{PSC}).
Thus it is of interest that the containment $I_\cL^{(3)}\subseteq I_\cL^2$ fails when $\cL$ is the Klein or Wiman configuration; in particular, the defining equation of the line configuration is in $I_\cL^{(3)}$ but not in $I_\cL^2$.  This was first confirmed computationally \cite{BNAL}, then proved conceptually in \cite{S14} in the case of the Klein configuration. We offer two new conceptual proofs based on representation theory which work for both configurations.

The \emph{resurgence} $$\rho(I) =\sup\left\{\frac{m}{r}: I^{(m)}\not\subseteq I^r\right\}$$ and \emph{asymptotic resurgence} $$\widehat\rho(I) = \sup\left\{\frac{m}{r}: I^{(mt)}\not\subseteq I^{rt}, t\gg0\right\}$$ were respectively introduced in \cite{BH10} and \cite{GHV13} to study failures of containment in more depth (see, e.g., \cite{DHNSST15}).  These invariants are closely related to Waldschmidt constants via the inequalities
\begin{equation}\label{GHvTbounds}
\frac{\alpha(I)}{\widehat{\alpha}(I)}\leq \widehat{\rho}(I)\leq \frac{\omega(I)}{\widehat{\alpha}(I)},
\end{equation}
and
\begin{equation}\label{BHbounds}
\widehat{\rho}(I)\leq \rho(I)\leq \frac{\operatorname{reg}(I)}{\widehat{\alpha}(I)};
\end{equation}
see \cite{BH10} and \cite{GHV13}.  Here $\omega(I)$ denotes the maximal degree of a generator in a minimal set of generators for $I$, and $\reg(I)$ is the regularity of $I$.  

The Klein and Wiman ideals $I_\cL$ each satisfy $\alpha (I_\cL) = \omega(I_\cL)$, and therefore by (\ref{GHvTbounds}) the computation of $\widehat\rho(I_\cL)$ is equivalent to the computation of $\widehat\alpha(I_\cL)$.

\begin{theorem}\label{thm-asympResIntro}
For the Klein configuration of lines, we have
$$1.230 \approx \frac{16}{13} \leq  \widehat \rho(I_{\cK}) \leq \frac{816}{661} \approx 1.234.$$ For the Wiman configuration of lines,  $$\widehat\rho(I_{\cW}) = \frac{32}{27} \approx 1.185.$$\end{theorem}

On the other hand, we compute the resurgence exactly for both configurations.

\begin{theorem}\label{thm-resIntro}
If $\cL = \cK$ or $\cW$, then $\rho(I_{\cL}) = \frac{3}{2}$.
\end{theorem}
For the proof (given at the end of Section \ref{sec-resurgence}), we show that the ideal $I_\cL$ is generated by three homogeneous forms of the same degree,
which allows us to compute the regularity of powers $I_{\cL}^r$ by results in \cite{NS16}.
Theorem \ref{thm-resIntro} follows easily
using this, together with $I_\cL^{(3)}\not\subseteq I_\cL^2$, containment results from \cite{BH10} and our knowledge of Waldschmidt constants.

\subsection*{Conventions} For simplicity we work over $\C$ for the majority of the paper, although it is likely that analogous results hold over other fields so long as the characteristic is sufficiently large.  In Section \ref{sec-posChar} we will briefly discuss the Klein configuration in characteristic 7, where some exceptional behavior occurs.

By a curve on a surface we usually mean an effective divisor.  We say a curve is $m$-uple at a point $p$ to mean that the multiplicity of the curve at $p$ is at least $m$.

\subsection*{Organization of the paper} In \S \ref{sec-prelim} we will recall the necessary definitions and the basic geometry of the Klein and Wiman configurations, as well as the group actions giving rise to them and the corresponding rings of invariants.  In \S \ref{sec-nef} we prove our upper bound on the Waldschmidt constants and indicate the correspondence between lower bounds on the Waldschmidt constants and nefness of divisors.  In \S \ref{sec-invariantSeries} we use some representation theory to study invariant linear series on the blowup $X_\cL$.  We precisely define the expected dimension of such a series and prove Theorem \ref{thm-expDimIntro}.  In \S\ref{sec-negKlein}-\ref{sec-negWiman} we study explicit negative curves on $X_\cL$ to prove Theorem \ref{thm-curvesIntro} and deduce Theorem \ref{thm-WaldschmidtIntro}.  We study the asymptotic resurgence and resurgence in \S\ref{sec-asympRes} and \S\ref{sec-resurgence}, respectively. We mention some results in characteristic $7$ in \S\ref{sec-posChar}.

\subsection*{Acknowledgements}
We would like to thank Izzet Coskun, Alex K\"uronya, Piotr Pokora, and Giancarlo Urz\'ua for many helpful discussions and Federico Galetto for his input on the second proof of Proposition \ref{prop-containmentFailure}.  We would also like to thank
the Mathematisches
Forschungsinstitut Oberwolfach for hosting workshops in February 2014 and March 2016
where some of the work presented in this paper was conducted.  Finally, we would like to thank the anonymous referees, whose comments greatly improved the paper.

\section{Preliminaries}\label{sec-prelim}

\subsection{Definitions and notation}\label{ssec-definitions}

For a line configuration $\cL$ in $\P^2$ we write $X_{\cL}$ for the blowup of $\P^2$ at the singular points in the configuration.  We write $H$ for the pullback of the hyperplane class. For each $m\geq 2$, we let $E_m$ be the sum of the exceptional divisors lying over the points in the configuration of multiplicity $m$.  We also write $I_{\cL}$ for the ideal of the singular points in the configuration.  We let $A_\cL$ be the divisor on $X_{\cL}$ given by the sum of the lines in the configuration.

\subsection{The Klein configuration of 21 lines}\label{ssec-prelimKlein}

Following \cite{BNAL,Elkies}, the Klein configuration $\cK$ is a configuration of $21$ lines in $\P^2$ whose intersections consist of
precisely $21$
quadruple points and $28$ triple points.  Thus, the divisor class of the line configuration on the blowup $X_\cK$ is $$A_\cK = 21H - 4E_4 - 3E_3,$$ and the intersection product on $X_\cK$ satisfies $$H^2 = 1 \qquad E_4^2 = -21 \qquad E_3^2 = -28,$$ where $H,E_4,E_3$ are pairwise orthogonal.  It is most natural to define the configuration over $\Q(\zeta)$, where $\zeta$ is a primitive $7$th root of unity.

Let $G=G_\cK$ be the unique simple group of order 168.  The group $G$ has an interesting irreducible $3$-dimensional representation $\rho$ over $\Q(\zeta)$.  There are generators $g,h,i$ such that this representation is given by
$$\rho(g) = \begin{pmatrix}
\zeta^4  & 0 & 0\\
0 & \zeta^2 & 0 \\
0 & 0 & \zeta
\end{pmatrix}, \quad
\rho(h) = \begin{pmatrix}
0 & 1 & 0\\
0 & 0 & 1 \\
1 & 0 & 0
\end{pmatrix}$$
and
$$\rho(i)=\frac{2\zeta^4+2\zeta^2+2\zeta+1}{7}\begin{pmatrix}
\zeta-\zeta^6 & \zeta^2-\zeta^5 & \zeta^4-\zeta^3 \\
\zeta^2-\zeta^5 & \zeta^4-\zeta^3 & \zeta-\zeta^6 \\
\zeta^4-\zeta^3 & \zeta-\zeta^6 & \zeta^2-\zeta^5
\end{pmatrix}.$$
Note that all three matrices have determinant 1 and the element $i$ has order $2$ (we also note that $(2\zeta^4+2\zeta^2+2\zeta+1)^2=-7$).
This representation gives an embedding of $G$ into $\SL_3(\Q(\zeta))$.  By projectivizing, $G$ acts on $\P^2$.

The transformation $\rho(i)$ has eigenvalues $1,-1,-1$.  The eigenspace for $-1$ is a
plane in $\C^3$, hence gives a line in $\P^2$ which is fixed pointwise by $\rho(i)$.
The orbit of this line under the action of $G$ consists of $21$ lines which comprise the Klein configuration $\cK$.
The eigenspace for $1$ is a point
$p\in\P^2;$
it is on exactly four of the lines so it is one of the quadruple points of the configuration.
Its orbit consists of all $21$ quadruple points of the configuration, so its stabilizer has order $8$.
The stabilizer turns out to be isomorphic to the dihedral group $D_8$ of order 8.
Its permutation representation on the 4 lines through the point $p$ is not faithful or transitive;
its image in the group $S_4$ of permutations of the 4 lines is isomorphic to $\Z/2\Z^{\times2}$.

The point $q=[1:1:1]\in \P^2$ is on $L$ and is a triple point of the configuration.
Its orbit is the set of all $28$ triple points of the configuration, and the stabilizer of the
point has order 6, isomorphic to $D_6\cong S_3$ (generated by $\rho(h)$ and $\rho(i)$).
It has a faithful permutation representation on the 3 lines through the point $q$.

\subsection{The Wiman configuration of 45 lines}\label{ssec-prelimWiman}
The Wiman configuration $\cW$ is a configuration of $45$ lines in $\P^2$ whose 201 intersections consist of
precisely 36 quintuple points, 45 quadruple points, and 120 triple points \cite{Wim96, BNAL}
(see also the table on p.~120 of \cite{Hir83}).  The divisor class of the line configuration on the blowup $X_\cW$ is therefore $$ A_\cW = 45H - 5E_5 - 4E_4 - 3E_3,$$ and the intersection product on $X_\cW$ satisfies $$H^2 = 1 \qquad E_5^2 = -36 \qquad E_4^2 = -45 \qquad E_3^2 = -120,$$ where $H,E_5,E_4,E_3$ are pairwise orthogonal.  The configuration is naturally defined over $\Q(\delta,\omega)$,
where $\delta^2=5$ and $\omega$ is a primitive $3$rd root of unity.

 The group $\PGL_3(\C)$ has a subgroup $G=G_\cW$ of order $360$ isomorphic to $A_6$.
 If we put $\mu_1 = (-1+\delta)/2$ and $\mu_2 = -(1+\delta)/2$, then  this subgroup is generated by transformations
 \begin{align*}
 R_1 &= \begin{pmatrix} 0 & 0 & 1 \\ 1 & 0 & 0 \\ 0 & 1 & 0 \end{pmatrix} &
 R_2 &= \begin{pmatrix} 1 & 0 & 0\\ 0 & -1 & 0 \\ 0 & 0 & -1\end{pmatrix} \\
 R_3 &= \frac{1}{2} \begin{pmatrix} -1 & \mu_2 &  \mu_1 \\ \mu_2 & \mu_1 & -1 \\ \mu_1 & -1 & \mu_2\end{pmatrix} &
 R_4 & = \begin{pmatrix} -1 & 0 & 0\\ 0 & 0 & -\omega^2 \\ 0 & -\omega & 0 \end{pmatrix}
 \end{align*}
 Note that while each of these transformations is actually in $\SL_3(\C)$, the subgroup of
 $\SL_3(\C)$ that they generate has order 1080 and is a triple cover of $A_6$, sometimes
 referred to as the \emph{Valentiner group} $\tilde G = 3\cdot A_6$. This group is a central extension of
 $A_6$ by $\Z/3\Z$; it contains in its center a subgroup isomorphic to $\Z/3\Z$ consisting of scalar matrices
with scalars the $3$rd roots of unity. The image of $\tilde G$ in $\PGL_3(\C)$ is $G$.





Looking at the eigenvectors of the involution $R_2$, it is easy to see that $R_2$ pointwise fixes the line $L$ with equation $L: x=0$.
The orbit of $L$ under $G$ consists of the $45$ lines in the Wiman configuration $\cW$.  The orbits and stabilizers of the singular points of the configuration ar as follows.
\begin{enumerate}
\item There are two $G$-orbits of size $60$ each consisting of triple points in the configuration.
The stabilizer of each of these points acts faithfully on the three lines through the point,
hence is isomorphic to the dihedral group $D_6\cong S_3$.
\item There is a single $G$-orbit of size $45$ consisting of quadruple points.
The stabilizer of each of these points turns out to be isomorphic to
the dihedral group $D_8$.
It acts on the 4 lines through the point, but not faithfully or transitively;
its image in the group $S_4$ of permutations of the four lines is ${\Z}/2{\Z}^{\times2}$.

\item There is a single $G$-orbit of size $36$ consisting of quintuple points.
The stabilizer of each of these points acts faithfully on the five lines through the point,
hence is isomorphic to the dihedral group $D_{10}$ (the only order 10 subgroup of $S_5$).
\end{enumerate}
Note that each of the 45 lines contains 16 points of the configuration, with four from each orbit.

\subsection{Invariants and the Klein configuration}\label{ssec-KleinInvariantsPrelim}

Most of the results in this paper rely on understanding the ring of invariant forms for the
action of the group $G$.  We recall the necessary facts from classical invariant theory here.
Consider $G=G_\cK \subset \SL_3(\C)$, the group of order $168$ defining the Klein configuration $\cK$.
Since $G$ is a subgroup of $\SL_3(\C)$, it acts in the natural way on the homogeneous coordinate ring
$S = \C[x,y,z]$ of $\P^2$.  Klein discovered the structure of the ring $S^G$ of polynomials invariant
under the action of $G$ \cite[\S 6]{Kle79}.  The ring $S^G$ is generated by invariant polynomials
$\Phi_4,\Phi_6,\Phi_{14}$, and $\Phi_{21}$, where $\Phi_d$ has degree $d$.  The invariant $\Phi_{21}=0$ defines the line configuration.  The polynomials
$\Phi_4,\Phi_6,\Phi_{14}$ are algebraically independent, but there is a relation
in degree $42$ between $\Phi_{21}^2$ and a polynomial in the other invariants.

The geometric significance of the invariants $\Phi_d$ is explained in Elkies \cite{Elkies}.  Briefly recalling the discussion there, we have $$\Phi_4 = x^3y+y^3z+z^3x,$$ so that $\Phi_4$ is the defining equation of the Klein quartic curve whose automorphism group is $G$.  The polynomial $\Phi_6$ can be taken to be $$\Phi_6 = -\frac{1}{54} H(\Phi_4) = xy^5 + yz^5+zx^5-5x^2y^2z^2,$$ where $H(\Phi_4)$ is the Hessian determinant
$$H(\Phi_4) := \begin{vmatrix} \partial^2 \Phi_4/\partial x^2&\partial^2 \Phi_4/\partial x \partial y&\partial^2 \Phi_4/\partial x \partial z\\
\partial^2 \Phi_4/\partial y \partial x&\partial^2 \Phi_4/\partial y^2&\partial^2 \Phi_4/\partial y \partial z\\
\partial^2 \Phi_4/\partial z\partial x&\partial^2 \Phi_4/\partial z \partial y&\partial^2 \Phi_4/\partial z^2\end{vmatrix}.$$

The degree $14$ invariant $\Phi_{14}$ is more complicated to describe; the graded piece $(S^G)_{14}$ is two-dimensional, so $\Phi_{14}$ is only uniquely defined mod $\Phi_4^2\Phi_6$.  One possible definition is that
$$\Phi_{14} = \frac{1}{9}BH(\Phi_4,\Phi_6),$$ where $BH(\Phi_4,\Phi_6)$ is the \emph{bordered Hessian} $$BH(\Phi_4,\Phi_6) := \begin{vmatrix} \partial^2 \Phi_4/\partial x^2&\partial^2 \Phi_4/\partial x \partial y&\partial^2 \Phi_4/\partial x \partial z & \partial \Phi_6 / \partial x\\
\partial^2 \Phi_4/\partial y \partial x&\partial^2 \Phi_4/\partial y^2&\partial^2 \Phi_4/\partial y \partial z & \partial \Phi_6/\partial y\\
\partial^2 \Phi_4/\partial z\partial x&\partial^2 \Phi_4/\partial z \partial y&\partial^2 \Phi_4/\partial z^2 & \partial \Phi_6/\partial z\\
\partial \Phi_6/\partial x & \partial \Phi_6/\partial y & \partial \Phi_6/\partial z & 0 \end{vmatrix}.$$
Finally, the invariant $\Phi_{21}$ is simply the product of the lines in the Klein configuration.  It can also be defined by a Jacobian determinant $$\Phi_{21} =\frac{1}{14} J(\Phi_4,\Phi_6,\Phi_{14}) = \frac{1}{14} \begin{vmatrix} \partial \Phi_4/\partial x & \partial \Phi_4/\partial y & \partial \Phi_4/\partial z\\
\partial \Phi_6/\partial x & \partial \Phi_6/\partial y & \partial \Phi_6/\partial z\\
\partial \Phi_{14}/\partial x & \partial \Phi_{14}/\partial y & \partial \Phi_{14}/\partial z\\ \end{vmatrix}$$ The degree $42$ relation between the invariants is given by the identity \begin{align}\label{phi21-rel} \Phi_{21}^2 &= \Phi_{14}^3-1728 \Phi_6^7 + 1008 \Phi_4 \Phi_6^4 \Phi_{14} +88 \Phi_4^2\Phi_6\Phi_{14}^2 +60032\Phi_4^3 \Phi_6^5\\ &\qquad + 1088\Phi_4^4\Phi_6^2\Phi_{14}-22016\Phi_4^6\Phi_6^3-256\Phi_4^7\Phi_{14}+2048\Phi_4^9\Phi_6.\notag\end{align} (Note that this relation differs from the one given in Elkies \cite{Elkies} due to an apparent error.)

Since $\Phi_4, \Phi_6$ and $\Phi_{14}$ are independent and $\Phi_{21}^2\in T=\C[\Phi_4, \Phi_6,\Phi_{14}]$,
the Veronese subring $(S^G)^{(42)}\subset T$ defined by $$(S^G)^{(42)} = \bigoplus_{k\geq 0} (S^G)_{42k}$$ is generated in degree $k=1$ by monomials in $\Phi_4,\Phi_6,\Phi_{14}$, subject only to the obvious relations.  This implies that the quotient $\P^2/G$ is isomorphic to the weighted projective space $\P(4,6,14)$.  The quotient map is given by \begin{align*}\phi:\P^2 &\to \P(4,6,14)\\ p&\mapsto [\Phi_4(p):\Phi_6(p):\Phi_{14}(p)]. \end{align*} The description of the union of lines $\Phi_{21}=0$ as the Jacobian determinant of $\Phi_4,\Phi_6,\Phi_{14}$ shows that $\Phi_{21}=0$ defines the ramification locus of $\phi$ away from points lying over the singular points $[0:1:0],[0:0:1]$ in $\P(4,6,14)$.  Note that the relation between $\Phi_{21}$ and the other invariants implies that the points lying over $[1:0:0]$ are in the line configuration.

The next lemma clarifies the relationship between $G$-invariant curves on $\P^2$ and $G$-invariant homogeneous forms.

\begin{lemma}\label{lem-invariantEqn}
For $G=G_\cK$, let $C\subset \P^2$ be a $G$-invariant curve which does not contain the Klein configuration $\cK$ of lines.  Then the defining equation $f\in S$ of $C$ is $G$-invariant and lies in the subalgebra $T = \C[\Phi_4,\Phi_6,\Phi_{14}]$ of $S$.
\end{lemma}
\begin{proof}
Since the ramification locus of $\phi$ consists of the union of the lines in the Klein configuration and finitely many points lying over the singularities in $\P(4,6,14)$, the map $\phi$ is a local isomorphism near a general point $p\in C$.  The curve $\phi(C)$ is defined by a single weighted homogeneous equation $g(w_0,w_1,w_2)=0$ in the coordinates  $w_0,w_1,w_2$ of the weighted projective space.  Then the pullback $\phi^*g$ of this equation defines $C$ and is in the subalgebra $T$.
\end{proof}

\begin{remark}\label{rem-KleinOrbit} We record here for later use the orbit sizes for the action of $G$ on $\P^2$, following \cite{Elkies}.
\begin{enumerate}
\item The triple points in the configuration form an orbit of size $28$.

\item The quadruple points in the configuration form an orbit of size $21$.

\item The invariant curves $\Phi_4 = 0$ and $\Phi_6=0$ meet in an orbit of $24$ points lying over the singular point $[0:0:1]\in\P(4,6,14)$.

\item The invariant curves $\Phi_4 = 0$ and $\Phi_{14}=0$ meet in an orbit of $56$ points lying over the singular point $[0:1:0]\in \P(4,6,14)$.

\item The invariant curves $\Phi_6 = 0$ and $\Phi_{14} = 0$ are tangent at an orbit of $42$ points lying over the singular point $[1:0:0]\in \P^2(4,6,14)$.  These points lie on the line configuration.

\item Any point on the line configuration not mentioned above has an orbit of size $84$.

\item Any point not mentioned above has an orbit of size $168$.
\end{enumerate}
\end{remark}

\subsection{Invariants and the Wiman configuration}\label{ssec-WimanInvariants}

The discussion of invariant forms for the action of $G=G_\cW\cong A_6$ on $\P^2$ which gives rise to the Wiman configuration is highly analogous to the case of the Klein configuration.  The main additional complication is that $G$ is only a subgroup of $\PGL_3(\C)$, so that it does not act on the homogeneous coordinate ring $S = \C[x,y,z]$ of $\P^2$.  We must therefore work with the Valentiner group $\tilde G \subset \SL_3(\C)$ of order $1080$, which has a natural action on $S$.

The ring of invariants $S^{\tilde G}$ is again fully understood by the theory of complex reflection groups.  The ring of invariants is generated by forms $\Phi_6$, $\Phi_{12}$, $\Phi_{30}$, and $\Phi_{45}$, where $\Phi_d$ has degree $d$.  The invariant $\Phi_{45}=0$ defines the line configuration.  Here $\Phi_6$, $\Phi_{12}$, and $\Phi_{30}$ are algebraically independent and $\Phi_{45}^2$ is a polynomial in the other invariants.

While $\Phi_6$ is uniquely determined up to scale, it does not have a particularly nice equation.  To compute it we recall the \emph{Reynold's operator} $R_G:S\to S^G$ for a group $G$
acting on a polynomial ring $S$ with its ring of invariants $S^G$, defined as
$$R_G(f) = \frac{1}{|G|}\sum_{g\in G} g(f).$$
Then we can compute $\Phi_6$ as
$$\Phi_6 = 16\, R_{\tilde G}(x^6),$$
where the coefficient 16 is chosen so that the coefficient of $x^6$ is 1.
Carrying this calculation out and choosing $\omega = e^{2\pi i/3}$ and $\delta  = -\sqrt 5$ gives
\begin{align*}
\Phi_6 &= x^6+y^6+z^6+3(5-\sqrt{15}\,i) x^2y^2z^2\\
&\quad+\frac{3}{4}(2\sqrt{5}-(5-\sqrt{5})\omega)(x^4y^2+y^4z^2+z^4x^2)\\
&\quad+\frac{3}{4} (5-\sqrt{5}+(5+\sqrt{5})\omega)(x^4z^2+y^4x^2+z^4y^2).
\end{align*}

The higher invariants $\Phi_{12},\Phi_{30},\Phi_{45}$ can be given by expressions
completely analogous to the invariants for the Klein configuration.  We can take
\begin{align*}
\Phi_{12} &= c_{12} H(\Phi_6)\\
\Phi_{30} &= c_{30} BH(\Phi_6,\Phi_{12})\\
\Phi_{45} &= J(\Phi_6,\Phi_{12},\Phi_{30})
\end{align*}
where we write $H,BH,J$ for the Hessian, bordered Hessian, and Jacobian determinants,
respectively (see \S\ref{ssec-KleinInvariantsPrelim}).  We choose the constants $c_d\in \C$ so
that the coefficient of $x^d$ in $\Phi_d$ is normalized to be $1$.  (Note that  $\Phi_{45}$ does not
have an $x^{45}$ term since $[1:0:0]$ is one of the quadruple points in the configuration; however, we will not work in any substantial way with $\Phi_{45}$ and
therefore do not worry about its normalization.) Up to scalars, we have
\begin{align*}
\Phi_{45}^2 & \sim 16\Phi_6^{13}\Phi_{12}-160\Phi_6^{11}\Phi_{12}^2+816\Phi_6^9\Phi_{12}^3-2188\Phi_6^7\Phi_{12}^4+3271\Phi_6^5\Phi_{12}^5\\
&\quad -1539 \Phi_6^3\Phi_{12}^6+351 \Phi_6\Phi_{12}^7 +72 \Phi_6^{10}\Phi_{30}-396 \Phi_6^8\Phi_{12}\Phi_{30}+954 \Phi_6^6\Phi_{12}^2\Phi_{30}\\
&\quad +99\Phi_6^4\Phi_{12}^3\Phi_{30} -1377 \Phi_6^2\Phi_{12}^4\Phi_{30}+243 \Phi_{12}^5\Phi_{30}+324\Phi_6^5\Phi_{30}^2-1944\Phi_6^3\Phi_{12}\Phi_{30}^2\\
&\quad +729 \Phi_6\Phi_{12}^2\Phi_{30}^2 + 729 \Phi_{30}^3.
\end{align*}

\begin{remark}
In the case of the Klein configuration the first two invariants $\Phi_4,\Phi_6$ were both uniquely determined up to scale,
but for the Wiman configuration there is a pencil of invariant forms of degree $12$ and a $4$-dimensional vector space
of invariant forms of degree $30$.  While the determinantal formulas for the invariants give one way of eliminating the
ambiguity in the choice of invariants, the ambiguity can also be naturally eliminated by looking at invariants that pass
through interesting points in the configuration.  We will investigate this further in Section \ref{sec-negWiman}.
\end{remark}

The quotient of $\P^2$ by $A_6$ is the weighted projective space $\P(6,12,30)$, with quotient map  \begin{align*}
\phi : \P^2 &\to \P(6,12,30)\\
p & \mapsto [\Phi_6(p):\Phi_{12}(p):\Phi_{30}(p)]
\end{align*}
Away from the preimages of the singular points in $\P(6,12,30)$, the ramification locus of $\phi$ is the line configuration $\Phi_{45}=0$. The relation between $\Phi_{45}^2$ and the other invariants implies that the points $[1:0:0]$ and $[0:1:0]$ in $\P(6,12,30)$ are both in the image of the line configuration; on the other hand, the points lying over $[0:0:1]$ form a single orbit of $72$ points cut out by $\Phi_6$ and $\Phi_{12}$.  A point in $\P^2$ with nontrivial stabilizer either lies on the line configuration or is one of these $72$ points.

The next lemma follows exactly as in the case of the Klein configuration.

\begin{lemma}\label{lem-invariantEqnWiman}
For $G=G_\cW$, let $C\subset \P^2$ be a $G$-invariant curve which does not contain the Wiman configuration $\cW$ of lines.  Then the defining equation $f\in S$ of $C$ is $\tilde G$-invariant and lies in the subalgebra $T = \C[\Phi_6,\Phi_{12},\Phi_{30}]$ of $S$.
\end{lemma}

\begin{remark}\label{rem-WimanOrbits} Here we record the orbit sizes for the action of  $A_6$ on $\P^2$, following \cite[p.18]{Crass}.
\begin{enumerate}
\item There are two orbits of $60$ triple points.
\item The $45$ quadruple points form an orbit.
\item The $36$ quintuple points form an orbit.
\item The curves $\Phi_6=0$ and $\Phi_{12}=0$ intersect in an orbit of $72$ points lying over $[0:0:1]\in \P(6,12,30)$.
\item The curves $\Phi_6 = 0$ and $\Phi_{30}=0$ are tangent at an orbit of $90$ points lying over $[0:1:0]\in \P(6,12,30)$.  These points are all on the line configuration.
\item Any point on the line configuration not mentioned above has an orbit of size $180$.
\item Any point not mentioned above has an orbit of size $360$.
\end{enumerate}
\end{remark}

\section{Nef divisors and the Waldschmidt constant}\label{sec-nef}

In this section we first bound the Waldschmidt constant for the Klein and Wiman configurations from above by constructing curves in symbolic powers of the ideal.  We then give an initial discussion of our strategy for bounding the Waldschmidt constant from below.

\begin{proposition}\label{prop-KleinUpper}
Let $I_\cK$ be the ideal of the $49$ points of the Klein configuration. Then
$$\widehat\alpha(I_\cK) \leq \frac{13}{2}.$$
\end{proposition}
\begin{proof}

For any integer $k\geq 1$ we define a divisor class $$D_k = (28k+2)H - 2kE_4 -5kE_3$$ on the blowup $X_\cK$.  Observe that the vector space dimension of the linear series $|D_k|$ is at least $${28k+4\choose 2} - 21{2k+1\choose 2} - 28{5k + 1 \choose 2} = 7k+6 >0.$$  Let $A_{\cK} = 21H - 4E_4 - 3E_3$ be the class of the union of the lines in $\cK$.  Then $$D_k + 3kA_\cK = (91k+2)H - 14kE_4- 14kE_3$$ is an effective divisor.  This gives an element of the symbolic power $I_\cK^{(14k)}$ of degree $91k+2$.  Letting $k\to \infty$ proves the proposition.
\end{proof}

\begin{proposition}\label{prop-WimanUpper}
For the ideal $I_{\cW}$ of the Wiman configuration, we have $$\widehat \alpha(I_\cW) \leq \frac{27}{2}.$$
\end{proposition}
\begin{proof}
The strategy is the same as in the proof of Proposition \ref{prop-KleinUpper}.  For $k\geq 1$, let $D_k$ be the divisor class
 $$D_k = (36k+6)H -kE_5 - 2k E_4 - 3k E_3$$ on $X_{\cW}$. Then the vector space dimension of the linear series $|D_k|$ is at least $${36k+8\choose 2} - 36 {k+1\choose 2} - 45 {2k+1\choose 2} - 120 {3k+1\choose 2} = 27k+ 28 >0.$$ Let $A_\cW = 45H - 5E_5 -4E_4 - 3E_3$ be the class of the union of the lines in $\cW$.
Then $$D_k + kA_\cW = (81k+6)H-6k E_5 -6k E_4 - 6k E_3,$$ giving an element of $I_{\cW}^{(6k)}$ of degree $81k+6$.  The result follows when $k\to \infty$.
\end{proof}

In the other direction, the proofs of Propositions \ref{prop-KleinUpper} and \ref{prop-WimanUpper} also suggest a method to establish lower bounds on the Waldschmidt constant.  The next two lemmas explain how we will approach this problem.

\begin{lemma}\label{lem-lowerKlein}
Let $k>0$ be a positive rational number, and let $D_k$ be the $\Q$-divisor class $$D_k=(28k+2)H-2kE_4 - 5kE_3$$ on the blowup $X_\cK$ of the points in the Klein configuration. Let $$D = 28H - 2E_4 - 5E_3.$$ If $D$ is nef, then $\widehat\alpha(I_{\cK}) = \frac{13}{2}$.  If $D_k$ is nef, then $$\widehat \alpha(I_{\cK}) \geq \frac{91k+24}{14k+4}$$
\end{lemma}

While we will not be able to show $D$ is nef, good bounds on the Waldschmidt constant $\widehat \alpha(I_\cK)$ can be obtained by showing $D_k$ is nef for large $k$.  It will be important later to notice that the divisor $D$ meets the class $A_\cK$ of the line configuration orthogonally: $D\cdot A_\cK=0$.  On the other hand, for $k>0$, we have $D_k\cdot A_\cK>0$.  Also observe that $D_k$ is effective by the proof of Proposition \ref{prop-KleinUpper}. Therefore, $D$ is pseudo-effective.

\begin{proof}
Suppose that $D_k$ is nef, and suppose there is a rational number $\beta$ such that  $$\widehat\alpha  (I_{\cK}) < \beta < \frac{91k+24}{14k+4}.$$ Then the $\Q$-divisor class $F = \beta H- E_4-E_3$ is effective.  However, any curve in a multiple $|mF|$ also contains the line configuration, since  $$F\cdot A_\cK = 21\beta-168<0.$$ Since $A_\cK^2 = -147$, if we strip off as many copies of $A_\cK$ from $F$ as possible we get the residual effective $\Q$-divisor $$F' = F - \frac{168-21\beta}{147}A_\cK =(4\beta-24)H-\frac{1}{7}(4\beta -25)E_4 -\frac{1}{7}(3\beta -17) E_3$$ which has $F'\cdot A_\cK = 0$.    We finally compute $$F'\cdot D_k = 28\beta k-182k+8\beta -48.$$ The inequality $\beta < (91k+24)/(14k+4)$ then implies $F'\cdot D_k <0$, contradicting that $D_k$ is nef.

If $D$ is nef, then $D_k$ is nef for every $k\geq 1$.  As $k\to \infty$, we find $\widehat\alpha(I_{\cK}) \geq \frac{13}{2}$.  Since $\widehat \alpha(I_{\cK}) \leq \frac{13}{2}$ by Proposition \ref{prop-KleinUpper}, we conclude that $\widehat\alpha(I_{\cK}) = \frac{13}{2}$.
\end{proof}

Since our computation of the Waldschmidt constant for the Wiman configuration will be sharp, the analogous lemma for the Wiman is easier.

\begin{lemma}\label{lem-lowerWiman}
If $D = 36H - E_5 - 2E_4 - 3E_3$ is nef on $X_{\cW}$, then $$\widehat{\alpha}(I_{\cW}) = \frac{27}{2}.$$
\end{lemma}

Note that $D^2=0$ and $D\cdot A_\cW = 0$.  Also, $D$ is pseudo-effective by the proof of Proposition \ref{prop-WimanUpper}.

\begin{proof}
Suppose $D$ is nef and that there is a rational number $\beta$ such that $$\widehat\alpha(I_\cW) < \beta < \frac{27}{2},$$ so that the $\Q$-divisor class $F = \beta H - E_5-E_4-E_3$ is effective.  Then $$F\cdot D = 36\beta -36-90-360 = 36\left(\beta - \frac{27}{2}\right) <0,$$ contradicting that $D$ is nef.  Therefore $\widehat\alpha(I_{\cW}) \geq \frac{27}{2}$, and equality holds by Proposition \ref{prop-WimanUpper}.
\end{proof}

\section{Invariant linear series}\label{sec-invariantSeries}

Our goal is to use Lemmas \ref{lem-lowerKlein} and \ref{lem-lowerWiman} to establish lower
bounds on the Waldschmidt constant for the Klein and Wiman configurations.  Let $G=G_\cL$ act on $X_\cL$.
 To use either
lemma, we must show some particular pseudo-effective, $G$-invariant divisor class $D$ on the blowup $X_\cL$ is nef.  While we will not need to directly apply the next lemma, it motivates our study of invariant curves of negative self-intersection.  The proof is straightforward, so we omit it.

\begin{lemma}
Suppose $D$ is a $G$-invariant divisor class on $X_\cL$ which is a limit of $G$-invariant effective $\Q$-divisors.  If $D$ is not nef, then there is a $G$-invariant, $G$-irreducible curve $B$ on $X_\cL$ such that $D\cdot B < 0$ and $B^2<0$.
\end{lemma}

Since the divisors appearing in Lemmas \ref{lem-lowerKlein} and \ref{lem-lowerWiman} intersect the class $A_\cL$ of the line configuration nonnegatively,
it is enough to study negative curves other than $A_\cL$.  Lemmas \ref{lem-invariantEqn}
and \ref{lem-invariantEqnWiman} tell us that the defining equation of any $G$-irreducible curve other than $A_\cL$ is a
polynomial in the fundamental invariant forms $\Phi_4,\Phi_6,\Phi_{14}$ if $\cL = \cK$
(resp. $\Phi_6,\Phi_{12},\Phi_{30}$ if $\cL = \cW$).  This motivates the next definition.

\begin{definition}\

\begin{enumerate}
\item[(1)] If $\cL = \cK$, let $T = \C[\Phi_4,\Phi_6,\Phi_{14}]\subset S$.  For integers $m_4,m_3 \geq 0$, we let $$ T_d(-m_4E_4-m_3E_3)\subset T_d$$ denote the subspace of forms of degree $d$ which are $m_4$-uple at the 21 quadruple points of $\cK$ and $m_3$-uple at the $28$ triple points of $\cK$.

\item[(2)] If $\cL = \cW$, let $T = \C[\Phi_6,\Phi_{12},\Phi_{30}]\subset S$.  For integers $m_5,m_4,m_3 \geq 0$, we let $$ T_d(-m_5E_5-m_4E_4-m_3E_3) \subset T_d$$ denote the subspace of forms of degree $d$ which are $m_5$-uple at the $36$ quintuple points of $\cW$, $m_4$-uple at the $45$ quadruple points of $\cW$, and $m_3$-uple at the $120$ triple points of $\cW$.
\end{enumerate}

\end{definition}

For example, for $\cL = \cK$, elements of the vector space $T_d(-m_4E_4-m_3E_3)$ define $G$-invariant curves in the linear series $|dH - m_4E_4 - m_3E_3|$ on $X_\cK$.

\begin{remark}
Since there are two orbits of $60$ triple points in $\cW$, it also makes sense to assign different multiplicities at the different orbits.  We will not need this more general construction, however.
\end{remark}

Several questions are immediate.  What is the \emph{dimension} of $T_d(-m_4E_4-m_3E_3)$?  Is there an \emph{expected} dimension for this series? When the series is nonempty, is there a $(G$-)\emph{irreducible} curve in the series?  In this section we propose a definition of the expected dimension which gives a lower bound on the actual dimension.  The other questions will be taken up in some specific cases in later sections.

\subsection{Leading terms of invariants}\label{ssec-leadingTerm}

In this subsection we prove general results about the leading term of an invariant form when expressed in local coordinates at a point $p\in \P^n$.  We set up our initial discussion in such a way that it will apply to both the Klein and Wiman configurations.  These results allow us to quantify the number of conditions required for an invariant form to have an $m$-uple point at one of the points in the configuration.

\subsubsection{Leading terms in general}  Let $p \in \P^n$ and let $S$ be the homogeneous coordinate ring of $\P^n$.  Suppose $\tilde G_p\subset \GL_{n+1}(\C)$ is a finite group which fixes $p$ and let $G_p$ be the image of $\tilde G_p$ in $\PGL_{n+1}(\C)$.  Then the kernel of $\tilde G_p\to G_p$ is cyclic of some order $m\geq 1$, generated by the scalar matrix $\omega I$ with $\omega  = e^{2\pi i/m}$.  If there is a $\tilde G_p$-invariant form $0\neq \Psi_d \in S_d$, this forces $m | d$.  On the other hand, if $d$ satisfies $m|d$, then the action of $\tilde G_p$ on $S_d$ descends to an action of $G_p$ on $S_d$ since $\omega I$ acts by the identity on $S_d$.

Let $I_p  \subset S$ be the homogeneous ideal of $p$.  Since $\tilde G_p$ fixes $p$, the powers $I_p^k$ are all $\tilde G_p$-invariant, so $\tilde G_p$ acts on the quotients $I_p^k / I_p^{k+1}$ and on their graded pieces $(I_p^k/I_p^{k+1})_d$.  If $m|d$, then $G_p$ also acts on $(I_p^k/I_p^{k+1})_d$.  Then the next lemma is obvious but crucial.

\begin{lemma}\label{lem-invariant}
Suppose $0\neq \Psi_d\in (I_p^k)_d$ is $\tilde G_p$-invariant (so $m|d$ and $k\leq d$).  Then the element $\overline \Psi_d \in (I_p^k/I_p^{k+1})_d$ is both $\tilde G_p$- and $G_p$-invariant.
\end{lemma}

Now let $(\OO_p,\frakm_p)$ be the local ring of $\P^n$ at $p$.  Then both $\tilde G_p$ and $G_p$ act on $\OO_p$ and the powers $\frakm_p^k$ are invariant, so that $\frakm_p^k/\frakm_p^{k+1}$ is both a $\tilde G_p$- and $G_p$-module.  To identify the $G_p$-modules $(I_p^k/I_p^{k+1})_d$ more geometrically, it is useful to compare them with the symmetric powers $$\Sym^k \frakm_p/\frakm_p^2 \cong \frakm_p^k/\frakm_p^{k+1}$$ of the cotangent space.

\begin{lemma}\label{lem-homogCotangent}
Let $W$ be the $1$-dimensional $\tilde G_p$-module $(S/I_p)_1$, and let $w\in S_1$ be a linear form not passing through $p$.  If $k\leq d$ then there is an isomorphism of $\tilde G_p$-modules \begin{align*}(I_p^k/I_p^{k+1})_d & \cong \frakm_p^k / \frakm_p^{k+1} \te W^{\otimes d} \\ F  & \mapsto  \frac{F}{w^d} \te w^d \end{align*}
\end{lemma}

Again the proof is clear.  In the situations of this paper we can further assume $d$ is such that $W^{\te d}$ is trivial.  We combine the observations in this subsection in the following  form.

\begin{corollary}\label{cor-leadingTerm}
Suppose $0\neq \Psi_d \in (I_p^k)_d$ is $\tilde G_p$-invariant and that $d$ is a multiple of the order of any linear character of $\tilde G_p$.  Let $w\in S_1$ be a linear form not passing through $p$.  Then the element $$\tilde \Psi_d := \Psi_d/w^d\in \frakm_p^k/\frakm_p^{k+1}$$ is $G_p$-invariant.  Thus if $\tilde \Psi_d\neq 0$ then it spans a trivial $G_p$-submodule of $\frakm_p^k/\frakm_p^{k+1}$
\end{corollary}
\begin{proof}
The assumptions on $d$ and Lemma \ref{lem-homogCotangent} show that there is an isomorphism $(I_p^k/I_p^{k+1})_d\cong \frakm_p^k/\frakm_p^{k+1}$ of both $\tilde G_p$-modules and $G_p$-modules, with $\Psi_d$ on the left corresponding to $\tilde \Psi_d$ on the right.  Then $\Psi_d$ is $G_p$-invariant by Lemma \ref{lem-invariant}, so $\tilde \Psi_d$ is also $G_p$-invariant.
\end{proof}

\begin{example}
For arbitrary group actions the conclusion of Corollary \ref{cor-leadingTerm} can fail without the assumption on $d$.  For example, let $p=[0:1]\in \P^1$ and let $\Z/2\Z = \tilde G_p = G_p$ act on the homogeneous coordinate ring of $\P^1$ by $x\mapsto x$, $y\mapsto -y$.  Then $x\in (I_p)_1$ is $G_p$-invariant, but $x/y\in \frakm_p/\frakm_p^2$ is not.
\end{example}

\subsubsection{Leading terms for the Klein and Wiman configurations}We next combine Corollary \ref{cor-leadingTerm} with some simple representation theory to heavily restrict the leading terms of an invariant form vanishing at a point in one of the line configurations.  For $\cL = \cK$ or $\cW$, we let $G$ and $\tilde G$ be the relevant groups (taking $\tilde G = G$ if $\cL = \cK$), and apply Corollary \ref{cor-leadingTerm} to the stabilizers $G_p$ and $\tilde G_p$ of a point $p$ in the configuration.

\begin{lemma}\label{lem-leadingTermConfig}
Let $\cL = \cK$ or $\cW$, and let $p\in \P^2$ be any point of the configuration.
\begin{enumerate}\item\label{keylem-part1} If $p$ is a point of multiplicity $n$ in $\cL$, then $G_p \cong D_{2n}$ and the $G_p$-module $U = \frakm_p/\frakm_p^2$ is irreducible of dimension $2$.   We have an isomorphism of $G_p$-modules $$\frakm_p^k/\frakm_p^{k+1} \cong \Sym^k U,$$ and the ring of invariants $(\Sym U)^{G_p}$ of the symmetric algebra is a polynomial algebra $\C[u,v]$ where $\deg u = 2$ and $\deg v = n$.

\item\label{keylem-part2} Fix a linear form $w$ not passing through $p$.  If $\Psi_d\in S_d$ is a $\tilde G$-invariant form of even degree which vanishes to order at least $k$
 at $p$, then $\tilde \Psi_d := \Psi_d/w^d \in \frakm_p^k/\frakm_p^{k+1}$ is $G_p$-invariant.

\end{enumerate}
\end{lemma}
\begin{proof}
(1) The fact that $G_p\cong D_{2n}$ was discussed in the preliminaries.  If $n\neq 4$, then the permutation representation of $G_p$ on the lines in $\cL$ through $p$ is faithful, and hence the action on $U$ is also faithful.  When $n=4$, the permutation representation is not faithful as the central element of $D_8$ acts trivially on the lines.  However, the central element acts on $U$ by multiplication by $-1$, so $U$ is still a faithful representation in this case.  If $U$ was not irreducible, then it would be a direct sum of $1$-dimensional representations and the image of $D_{2n}$ in $\GL(U)$ would be abelian.  Since $U$ is faithful, we conclude that it is irreducible.  The displayed isomorphism is obvious.  The computation of the ring of invariants of $\Sym U$ is well-known; see \cite{BG70} or \cite{ST54}.

(2) For the Klein configuration, we have $\tilde G_p = G_p =D_{2n}$ for $n=3$ or 4, and all linear characters of $\tilde G_p$ have order dividing $2$.  For the Wiman configuration, we have $\tilde G_p = D_{2m}\times \Z/3\Z$ since there are no nontrivial central extensions of $D_{2m}$ by $\Z/3\Z$ for $3\leq m\leq 5$.  Then the values of the linear characters of $\tilde G_p$ are $6$th roots of unity.  Any $\tilde G$-invariant form $\Psi_d$ of even degree has degree divisible by $6$ (see \S \ref{ssec-WimanInvariants}).  In either case, the result follows from Corollary \ref{cor-leadingTerm}.
\end{proof}

The next corollary is an immediate consequence.  It is a surprisingly powerful tool for determining explicit equations of invariants with prescribed multiplicities.  See \S\ref{sec-negKlein} and \ref{sec-negWiman} for applications.

\begin{corollary}\label{cor-parallel}
Let $\cL = \cK$ or $\cW$, let $p\in \P^2$ be a point of the configuration, and let $w$ be a linear form not passing through $p$.  If $\Psi_d\in S_d$ is a $\tilde G$-invariant form of even degree which vanishes at $p$, then it vanishes to order at least $2$ at $p$, and
$\tilde \Psi_d = \Psi_d/w^d$ lies in the $1$-dimensional trivial $G_p$-submodule of $\frakm_p^2/\frakm_p^3$.
\end{corollary}

\subsection{Expected dimension}\label{ssec-expDim}

Here we use Lemma \ref{lem-leadingTermConfig} to count the number of (not necessarily independent) linear conditions it is for an invariant form to have assigned multiplicities at the points in either the Klein or Wiman configurations.

\begin{definition}
We let $\cond_n(m)$ be the number of monomials of degree less than $m$ in a polynomial algebra $\C[u,v]$ where $\deg u = 2$ and $\deg v = n$.

\begin{enumerate} \item If $\cL = \cK$, then the \emph{expected dimension} $\edim(T_d(-m_4E_4-m_3E_3))$ is $$\max\{\dim T_d - \cond_4(m_4) - \cond_3(m_3),0\}.$$

\item If $\cL = \cW$, then the \emph{expected dimension}  $\edim(T_d(-m_5E_5-m_4E_4-m_3E_3))$ is $$\max\{\dim T_d - \cond_5(m_5)-\cond_4(m_4)-2\cond_3(m_3),0\}$$
\end{enumerate}
\end{definition}

(Recall that in the case of the Wiman configuration there are two orbits of triple points.)  We can now prove our main result in this section.

\begin{theorem}\label{thm-expDim}
If $\cL = \cK$, then $$\dim (T_d(-m_4E_4-m_3E_3)) \geq \edim(T_d(-m_4E_4-m_3E_3)) .$$ The analogous result holds for $\cL = \cW$.
\end{theorem}
\begin{proof}
Let $V\subset T_d$ be any subspace.  Fix an $n$-uple point in the configuration $p\in \cL$ with stabilizer $D_{2n}$, and fix a linear form $w$ not passing through $p$.  For $k\geq 0$, write $V_k\subset V$ for the subspace of forms which are at least $k$-uple at $p$.  By Lemma \ref{lem-leadingTermConfig} (\ref{keylem-part2}), the map \begin{align*}
r_k: V_k &\to \frac{\frakm_p^k}{\frakm_p^{k+1}}\\
\Psi_d &\mapsto \Psi_d/w^d
\end{align*}
has image contained in the subspace $(\frakm_p^k/\frakm_p^{k+1})^{G_p}$ of invariants, and its kernel is $V_{k+1}$.  Therefore $$\dim V_{k+1} = \dim \ker r_k \geq \dim V_k - \dim(\frakm_p^k/\frakm_p^{k+1})^{G_p}.$$  Then the subspace $V_m\subset V$ has codimension at most $\cond_{n}(m)$ by Lemma \ref{lem-leadingTermConfig} (\ref{keylem-part1}).

The theorem is proved by starting from $V= T_d$ and applying the above discussion once for each orbit of points in the configuration.\end{proof}

We conclude the section by investigating some of the immediate consequences of the theorem, as well as by indicating how to compute the terms in the formula for the expected dimension.

\begin{example}
We record some small values of $\cond_n(m)$ for easy access.
$$\begin{array}{c|cccc}
m & \cond_3(m) & \cond_4(m) & \cond_5(m) &\\
\hline 1  & 1 & 1 & 1\\
2 & 1 & 1& 1\\
3 & 2 & 2 & 2\\
4 & 3 & 2 & 2\\
5 & 4 & 4 & 3\\
6 & 5 & 4 & 4\\
7 & 7 & 6 & 5\\
8 & 8 & 6 & 6\\
\end{array}$$
\end{example}

\begin{example}
To aid in the computation of expected dimensions, note that for $\cL = \cK$ the dimension of the vector space $T_d$ is the coefficient of $t^d$ in the Taylor expansion of the rational function $$\frac{1}{(1-t^4)(1-t^6)(1-t^{14})} =1+ t^4+t^6+t^8+t^{10} + 2t^{12}+2t^{14}+2t^{16}+3t^{18}+\cdots.$$ A similar formula holds for the Wiman configuration. Similarly, $\cond_n(m)$ is the coefficient of $t^{m}$ in the Taylor expansion of the rational function $$\frac{t}{(1-t)(1-t^2)(1-t^n)}.$$
\end{example}

\begin{example}\label{ex-KleinNeg}
On $X_\cK$, we have $\dim T_{18} = 3$.  Therefore $T_{18}(-4E_4)$ has expected dimension $1$, and there is an effective invariant curve of class $18H-4E_4$.  It has self-intersection $-12$.

Similarly, $\dim T_{42} = 9$, so $T_{42}(-8E_3)$ has expected dimension $1$.  Therefore there is an invariant curve of class $42H-8E_3$.  It has self-intersection $-28$.  We will study this curve in more detail in Section \ref{sec-negKlein} to give our best bound on $\widehat\alpha(I_{\cK})$ that doesn't use substantial computer computations.
\end{example}

\begin{example}\label{ex-WimanUnexpected}
On $X_{\cW}$, we have $\dim T_{90} = 18$.  Therefore $T_{90}(-4E_4-8E_3)$ has expected dimension $0$.  However, we will see in Section \ref{sec-negWiman} that there is actually a unique $G$-irreducible curve of class $90H-4E_4-8E_3$; it has self-intersection $-300$.  The ``local'' linear conditions at each of the orbits of points in the configuration are not independent.  Studying this unexpected curve in detail will allow us to compute $\widehat\alpha(I_{\cW})$ exactly.
\end{example}

\begin{example}\label{ListOfNegs}
We can use Theorem \ref{thm-expDim} to find many additional interesting effective classes of
negative self-intersection on $X_\cK$.  We generate a list of effective divisor classes
$C_0, C_1,C_2,\ldots$ of negative self-intersection which meet all other classes on the list
nonnegatively.  This list further has the property that any class $D=T_d(-m_4E_4-m_3E_3)$ with negative self-intersection
and positive expected dimension with degree $0< d \leq 135786$ and $m_i\geq 0$ meets one of the curves on the list with smaller degree negatively.
\begin{align*}
C_0 &= 21H - 4E_4 - 3E_3 \\
C_1 &= 18H -4E_4 -0E_3                \\
C_2 &= 42H -0E_4 -8E_3              \\
C_3 &= 144H -4E_4 -27E_3             \\
C_4 &= 804H -28E_4 -150 E_3          \\
C_5 &= 2706H -100E_4 -504  E_3       \\
C_6 &= 7728H -288E_4 -1439    E_3    \\
C_7 &= 40992H -1534E_4 -7632  E_3    \\
C_8 &= 135786H -5088E_4 -25280 E_3   \\
C_9 &= 386880H -14500E_4 -72027  E_3 \\
C_{10} &= 2049732H -76828E_4 -381606 E_3\\
C_{11} &= 6787218H -254404E_4 -1263600 E_3
\end{align*}
Every class $C_i$ with $i\geq 1$ has expected dimension $1$ (note that the expected dimension of $C_0$ has not been defined).  There are far more open questions than settled ones here.  Can this list be extended infinitely?  Does every $G$-invariant curve of negative self intersection eventually appear on the list?  Are these classes representable by $G$-irreducible curves?
\end{example}

\begin{example}\label{ex-multipleCurve}
Notice that for the Klein configuration the series $T_{42}(- 8E_4-6E_3)$ consists of the divisor $2A_\cK$, where $A_{\cK}$ is the line configuration.  However, the expected dimension is $0$.
\end{example}

Example \ref{ex-multipleCurve} shows that if negative curves are in the base locus then the expected dimension can differ from the actual dimension.
 Computer calculations
which we have carried out in the Klein case suggest that equality holds on the Klein blowup unless there
is a negative curve in the base locus.  We thus formulate the following SHGH-type conjecture.

\begin{conjecture}
Let $D = dH - m_4E_4 - m_3E_3$ be a divisor on $X_\cK$.  If $D.C\geq 0$ for every $G$-invariant, $G$-irreducible curve $C$ of negative self-intersection with degree less than $d$, then $$\edim(T_d(-m_4E_4-m_3E_3)) = \dim (T_d(-m_4E_4-m_3E_3)).$$
\end{conjecture}

\begin{remark} The conjecture has been checked by computer when $d<144$.
First we computed the list of negative curves of degree less than $144$; see Example \ref{ListOfNegs} and Theorem \ref{thm-D7}.
Then we checked that $\dim T_d(-m_4 E_4 - m_3 E_3) = \edim T_d(-m_4 E_4 - m_3 E_3)$ whenever the multiplicities are {\it critical},
meaning that either
\begin{enumerate}\item increasing either of the multiplicities would either make the series intersect a negative curve negatively or make $\edim=0$, or
\item $\edim=0$, but decreasing either of the multiplicities makes the $\edim$ positive.\end{enumerate}
Note that if a non-critical series of invariants has $\edim > 0$ and $\dim \ne \edim$,
then increasing the multiplicities to get a critical series with $\edim > 0$
will give a series with $\dim \ne \edim$.
There are then not that many series to check, and the function $\texttt{series({d,m,n})}$ in the Supplementary Material runs quickly enough to compute the necessary dimensions in a couple hours on an ordinary desktop computer.
\end{remark}

\section{Negative invariant curves on $X_\cK$}\label{sec-negKlein}

In this section we study the curve $B$ of class $42H-8E_3$ on $X_\cK$ which was first discussed and proved to exist in Example \ref{ex-KleinNeg}.  Our goal is to prove the following theorem.

\begin{theorem}\label{thm-negativeCurveKlein}
There is a unique curve $B$ of class $42H - 8E_3$ on $X_\cK$.  It is $G$-invariant, $G$-irreducible, and reduced.
\end{theorem}

The main difficulty is to show that this curve is $G$-irreducible; this will require that we find its precise equation.  To make this computation tractable, we make heavy use of the results of Section \ref{ssec-leadingTerm}. The $G$-irreducibility of this curve has the following application to Waldschmidt constants.  Recall the definition of the divisor class $D_k = (28k+2)H - 2kE_4 - 5kE_3$ from Lemma \ref{lem-lowerKlein}.

\begin{corollary}\label{cor-KleinLowerThy}
The divisor $D_{16/7}$ is nef on $X_\cK$.  Therefore $$6.444 \approx \frac{58}{9}  \leq \widehat\alpha(I_\cK) \leq 6.5 .$$
\end{corollary}
\begin{proof}
Let $A_\cK$ be the class of the line configuration.  By Theorem \ref{thm-negativeCurveKlein}, the divisor class $$8A_\cK + 7B = 7D_{16/7}$$ is effective.  Both $A_\cK$ and $B$ are $G$-irreducible, so since $A_\cK\cdot D_{16/7} >0$ and $B\cdot D_{16/7}>0$ we conclude that $D_{16/7}$ is nef.  The inequalities follow from Lemma \ref{lem-lowerKlein} and Proposition \ref{prop-KleinUpper}.
\end{proof}

We will close the section with an indication of how to improve the bound with substantial computer computations.

\subsection{An alternate set of invariants.}

The equation of the curve $B$ is most naturally described in terms of an alternate set of invariants $\Psi_4,\Psi_6,\Psi_{12},\Psi_{14},$ where $\Psi_d$ has degree $d$.  These new invariants are defined by incidence conditions with respect to the triple points in $\cK$.  While the degree $4$ and $6$ invariants are uniquely determined up to scale, there are pencils of degree $12$ and degree $14$ invariants, spanned by $\langle \Phi_4^3,\Phi_6^2\rangle$ and $\langle \Phi_4^2\Phi_6,\Phi_{14}\rangle$, respectively.  We let $\Psi_{12}$ and $\Psi_{14}$ be the unique (up to scale) invariants passing through a triple point $p\in \P^2$ of the configuration $\cK$.  For clarity and to make the computation as conceptual as possible, we do not worry about the particular multiples of the invariants until later.  By Corollary \ref{cor-parallel}, $\Psi_{12}$ and $\Psi_{14}$ are actually both double at $p$.  Furthermore, in local affine coordinates centered at $p$, their leading terms are proportional.

Let $\tilde x,\tilde y$ be affine local coordinates centered at $p$, so that $\frakm_p^k$ is identified with $(\tilde x,\tilde y)^k$.  Let $w$ be a linear form not passing through $p$, and write $\tilde \Psi_d = \Psi_d/w^d \in \OO_p$.  Then we can find elements $A_i,B_i,C_i,D_i\in\C[\tilde x,\tilde y]$ which are homogeneous of degree $i$ such that  \begin{align*} \tilde \Psi_{14} &\equiv A_2 + A_3  \pmod{ \frakm_p^4}\\
\tilde \Psi_{12} &\equiv B_2 + B_3 \pmod{\frakm_p^4} \\
\tilde \Psi_6 &\equiv C_0 + C_1 \pmod{\frakm_p^2} \\
\tilde \Psi_4 &\equiv D_0 + D_1 \pmod{\frakm_p^2}.
\end{align*}
By Corollary \ref{cor-parallel}, there are constants $\mu,\nu\in \C$ such that \begin{align*}A_2 &= \mu B_2 \\ C_0 &=\nu D_0. \end{align*}  Furthermore, the invariant $C_0^2\Psi_4^3-D_0^3\Psi_6^2$ vanishes at $p$ and so must be double at $p$.  Therefore $$0\equiv C_0^2 \tilde \Psi_4^3-D_0^3\tilde \Psi_6^2 \equiv 3C_0^2D_0^2D_1-2C_0C_1D_0^3 \equiv \nu D_0^4(3\nu D_1-2 C_1) \pmod {\frakm_p^2}$$ which gives a relation $$C_1 = \frac{3}{2}\nu D_1.$$

\begin{remark}\label{rem-degree0}
The constants $\mu,\nu,D_0$ depend on the choice of linear form $w$ and the choice of a particular triple point $p$.  However, if we view $\mu,\nu,D_0$ as having degrees $2,2,4$ respectively, then degree $0$ homogeneous expressions in these constants do not depend on these choices (although they do depend on the particular normalizations of the invariants $\Psi_d$).  For example, $\nu D_0/\mu^3$ is the unique constant $\alpha\in\C$ such that $$\alpha \Psi_{14}^3 \equiv \Psi_6\Psi_{12}^3 \pmod {I_p^7}$$ where $I_p$ is the homogeneous  ideal of $p$, and applying the group action gives the same identity for any other choice of triple point.  We will abuse notation and write for example $$\left(\frac{\Psi_6\Psi_{12}^3}{\Psi_{14}^3}\right)(p):=\frac{\nu D_0}{\mu^3}$$ when we wish to emphasize the intrinsic nature of such constants.  While the individual constants e.g. $\mu$ are typically horrendous, such degree $0$ combinations are frequently very simple.
\end{remark}

\subsection{Equation of the curve of class $42H - 8E_3$}  For constants $\lambda_i\in \C$, we consider the curve  defined by $$\Psi_{42}:=\lambda_1 \Psi_{14}^3+\lambda_2\Psi_4\Psi_{12}^2\Psi_{14}+\lambda_3 \Psi_6\Psi_{12}^3=0.$$  When the constants $\lambda_i$ are chosen appropriately, the curve $\Psi_{42}=0$ will be the curve $B$ that we are searching for.  We now determine the correct constants $\lambda_i$.

\begin{lemma}\label{lem-7upleKlein}
The curve $\Psi_{42}=0$ is $7$-uple at $p$ if and only if
$$ \mu^3 \lambda_1 +  \mu D_0 \lambda_2+  \nu D_0 \lambda_3=0.$$
\end{lemma}
\begin{proof}
We expand the expression for $\tilde \Psi_{42}$, working mod $\frakm_p^7$.  We have
\begin{align*}
\tilde \Psi_{42} &= \lambda_1\tilde\Psi_{14}^3+\lambda_2\tilde\Psi_4\tilde\Psi_{12}^2\tilde\Psi_{14}+\lambda_3\tilde\Psi_6\tilde\Psi_{12}^3\\
&= \lambda_1 A_2^3+\lambda_2 D_0 B_2^2A_2+\lambda_3C_0B_2^3\\
&= (\lambda_1 \mu^3+\lambda_2 \mu D_0 + \lambda_3 \nu D_0  )B_2^3,
\end{align*}
from which the result follows.
\end{proof}

The next computation is similar albeit slightly more complicated.

\begin{lemma}\label{lem-8upleKlein}
The curve $\Psi_{42} = 0$ is $8$-uple at $p$ if it is $7$-uple at $p$ and $$2\mu\lambda_2+3\nu\lambda_3 =0.$$  More intrinsically, the curve $\Psi_{42}=0$ is $8$-uple at $p$ if the $\lambda_i$ satisfy the system\begin{align*}
\lambda_1 +\frac{1}{3}\left(\frac{\Psi_4\Psi_{12}^2}{\Psi_{14}^2}\right)(p) \cdot \lambda_2 &= 0 \\
\lambda_2 + \frac{3}{2}\left(\frac{\Psi_6\Psi_{12}}{\Psi_4\Psi_{14}}\right)(p)\cdot \lambda_3&= 0.
\end{align*}
\end{lemma}
\begin{proof}
Suppose the curve is $7$-uple at $p$.  We collect the degree $7$ terms in the expansion of $\tilde \Psi_{42}$ as follows, working mod $\frakm_p^8$.
\begin{align*}
\tilde \Psi_{42} &= \lambda_1 \tilde \Psi_{14}^3 + \lambda_2 \tilde \Psi_4 \tilde \Psi_{12}^2\tilde \Psi_{14} + \lambda_3 \tilde \Psi_6 \tilde \Psi_{12}^3\\
&= \lambda_1(3A_2^2A_3)+\lambda_2(D_0B_2^2A_3+2D_0B_2B_3A_2+D_1B_2^2A_2)\\&\quad+\lambda_3(3C_0B_2^2B_3+C_1B_2^3)\\ &= \lambda_1(3\mu^2B_2^2A_3)+\lambda_2(D_0B_2^2A_3+2\mu D_0B_2^2B_3 +\mu D_1B_2^3)\\&\quad+\lambda_3(3\nu D_0 B_2^2B_3+\frac{3}{2}\nu D_1B_2^3).
\end{align*} Divide this expression by the common factor $B_2^2$ and then collect the coefficients of $A_3$, $B_3$, and $D_1B_2$ to see that if the $\lambda_i$ satisfy the system
$$\begin{pmatrix} 3\mu^2 & D_0 & 0\\ 0 & 2\mu D_0 & 3\nu D_0\\ 0 & \mu & \frac{3}{2}\nu \\ \mu^3 & \mu D_0 & \nu D_0\end{pmatrix} \begin{pmatrix}\lambda_1\\\lambda_2\\\lambda_3\end{pmatrix} = 0$$ then $\Psi_{42}$ is $8$-uple at $p$ (the fourth equation here is the requirement that $\Psi_{42}$ be $7$-uple at $p$, by Lemma \ref{lem-7upleKlein}).  This matrix has rank $2$, from which the first part of the result follows.

 The second part of the result follows since the above system of 4 equations is equivalent to the system consisting of the 1st and 3rd equations.  Dividing through to obtain coefficients which are homogeneous of degree 0 (see Remark \ref{rem-degree0}) proves the second statement.
 \end{proof}

Having found the linear conditions which must be satisfied for $\Psi_{42}$ to be $8$-uple at $p$, we now fix specific multiples of the invariants $\Psi_d$ in order to compute the explicit equation. We put \begin{align*}
\Psi_4 &= \frac{2}{3} \Phi_4\\
\Psi_6 &= 2 \Phi_6\\
\Psi_{12} &= 2\Psi_4^3 - \Psi_6^2\\
\Psi_{14} &= \frac{1}{11}\Phi_{14} - \frac{8}{33} \Phi_4^2 \Phi_6.
\end{align*}
where the $\Phi_d$ are the standard invariants of Section \ref{ssec-KleinInvariantsPrelim}.
If $p= [1:1:1]\in \P^2$ is a triple point in $\cK$ then $$\phi(p) := [\Phi_4(p):\Phi_6(p):\Phi_{14}(p)] = [3:-2:-48],$$ from which we see that the above invariants $\Psi_d$ have the required incidence properties.

\begin{corollary}\label{cor-KleinEqn}
If the invariants $\Psi_d$ are normalized as above, then the curve $$2\Psi_{14}^3-3\Psi_4\Psi_{12}^2\Psi_{14}+\Psi_{12}^3\Psi_6=0$$ is $8$-uple at a triple point $p\in \cK$.
\end{corollary}
\begin{proof}
A straightforward computation shows that
$$\left(\frac{\Psi_4\Psi_{12}^2}{\Psi_{14}^2}\right)(p) = \left(\frac{\Psi_6\Psi_{12}}{\Psi_4\Psi_{14}}\right)(p)=  2$$
 Picking $\lambda_3=1$, the explicit equation follows from Lemma \ref{lem-8upleKlein}. \end{proof}

\subsection{$G$-irreducibility of the curve of class $42H-8E_3$}

Now that we have the equation of the curve of class $42H-8E_3$, the proof of Theorem \ref{thm-negativeCurveKlein} is easy.

\begin{proof}[Proof of Theorem \ref{thm-negativeCurveKlein}]
Consider the curve $B$ in $\P^2$ defined by the equation
$$2\Psi_{14}^3-3\Psi_4\Psi_{12}^2\Psi_{14}+\Psi_6\Psi_{12}^3=0,$$
where the invariants are normalized as in Corollary \ref{cor-KleinEqn}.  We will make  use of the modified quotient map \begin{align*} \psi:\P^2 &\to \P(4,6,14)\\ p&\mapsto [\Psi_4(p),\Psi_6(p),\Psi_{14}(p)].\end{align*}

To see that $B$ is $G$-irreducible, it is enough to see that the curve $B'$ in $\P(4,6,14)$ defined by the equation $$F(w_0,w_1,w_2):=2w_2^3-3(2w_0^3-w_1^2)^2w_0w_2+(2w_0^3-w_1^2)^3 w_1=0$$ is irreducible, since then any irreducible component of $B$ dominates $B'$.  If $B'$ is not irreducible, then there is a factorization of the form $$F = (F_2 w_2^2+F_1 w_2 + F_0)(G_1 w_2 + G_0)$$ where the $F_i,G_i\in \C[w_0,w_1]$ are weighted homogeneous of appropriate degrees to make the factors weighted homogeneous.  Comparing coefficients, $F_2G_1=2$, so $F_2$ and $G_1$ are both constant.  Therefore $\deg G_0 = 14$ and $G_0$ divides $(2w_0^3-w_1^2)^3w_1$.  But $2w_0^3-w_1^2$ is irreducible of degree $12$, so this is clearly impossible.
Therefore $B'$ is irreducible.

Note that $\psi$ is unramified over a general point in $B'$ since $B'$ is not the branch divisor, so $B$ is reduced since $B'$ is.

Finally, to see that $B$ is unique, consider the complete linear series $|42H-8E_3|$ on $X_\cK$.  Since $B^2 < 0$ on $X_{\cK}$, there is a curve in the base locus of this linear series.  Since the divisor class $42H-8E_3$ is $G$-invariant, its base locus is also $G$-invariant.  But then since $B$ has a single orbit of irreducible components, it follows that the only curve in the series is $B$.
\end{proof}

\subsection{Computer calculations}

To show that a divisor class $D_k = (28k+2)H-2kE_4-5kE_3$ is nef, one approach is to classify all $G$-invariant, $G$-irreducible curves on $X_\cK$ of negative self-intersection of degree $\leq 28k+2$ and verify that they meet $D_k$ nonnegatively.  A computer can carry out this computation in small degrees.  See the Supplementary Material for the methods used.

\begin{theorem}\label{thm-D7}
The only $G$-invariant, $G$-irreducible curves of negative self-intersection on $X_\cK$ with degree $\leq 200$ are of class $21H - 4E_4 - 3E_3$, $18H - 4E_4$, $42H - 8E_3$, and $144H - 4E_4 - 27 E_3$.  Therefore $D_7$ is nef, and $$6.480\approx\frac{661}{102}\leq \widehat\alpha(I_{\cK}) \leq 6.5.$$
\end{theorem}

In light of our computational evidence, the following conjecture seems reasonable.

\begin{conjecture}\label{conj-Klein}
The divisor $D = 28H - 2E_4 - 5E_3$ on $X_\cK$ is nef.  Therefore $$\widehat\alpha(I_\cK) = \frac{13}{2}.$$
\end{conjecture}

\section{A negative invariant curve on $X_\cW$}\label{sec-negWiman}

Here we prove that for the Wiman configuration we have $\widehat \alpha(I_{\cW}) = \frac{27}{2}$.  As with the Klein configuration, the computation relies on finding a single interesting invariant curve of negative self-intersection.  While the curve we studied for the Klein configuration was guaranteed to exist since the expected dimension of the series was positive, in this case the expected dimension is $0$ and the existence of the curve is quite surprising.  Our main focus of the section is to prove the following theorem.

\begin{theorem}\label{thm-negativeCurveWiman}
There is a unique curve $B$ of class $90H-4E_4 - 8E_3$ on $X_{\cW}$.  It is $G$-invariant, $G$-irreducible, and reduced.
\end{theorem}

The computation of the Waldschmidt constant is an immediate corollary.

\begin{corollary}\label{cor-WimanWaldschmidt}
The divisor $D= 36H - E_5 -2E_4 - 3E_3$ on $X_\cW$ is nef.  Therefore $$\widehat{\alpha}(I_{\cW}) = \frac{27}{2}.$$
\end{corollary}
\begin{proof}
Let $A_{\cW}$ be the class of the line configuration.  By Theorem \ref{thm-negativeCurveWiman}, the divisor class $$2A_\cW+3B = 10D$$ is effective.  Both $A_\cW$ and $B$ are $G$-irreducible.  Then since $D\cdot A_\cW=D\cdot B=0$, we conclude $D$ is nef.  Lemma \ref{lem-lowerWiman} completes the proof.
\end{proof}

As with the case of the Klein configuration, we begin by determining the explicit equation of the curve.  We then use the equation to prove $G$-irreducibility, which is somewhat more involved in this case.

\subsection{An alternate set of invariants} As with the Klein configuration, the equation of the curve $B$ is most easily described in terms of a different set of invariants defined by incidence properties with the points in the configuration.  Let $p_4,p_3,\overline{p}_3$ be a quadruple point and two triple points in different $G$-orbits.  We consider invariants $\Psi_6,\Psi_{12},\Psi_{24},\Psi_{30}$ specified by the following incidence conditions.  There is a pencil of invariant forms of degree $12$, and we let $\Psi_{12}$ pass through $p_4$.  There is a $3$-dimensional vector space of invariant forms of degree $24$, and we choose $\Psi_{24}$ to pass through $p_3$ and $\overline p_3$.  Finally, there is a $4$-dimensional vector space of invariant forms of degree $30$, and we choose $\Psi_{30}$ to pass through all three points $p_4,p_3,\overline p_3$.

Fix a linear form $w$ meeting none of the points in the configuration, and put $\tilde \Psi_d = \Psi_d/w^d$.  As with the Klein configuration, Corollary \ref{cor-parallel} shows that when we express the functions $\tilde \Psi_d$ in affine local coordinates around $p_3$ or $\overline p_3$, we get expansions \begin{align*}
\tilde \Psi_{30} &\equiv A_2+A_3 \pmod{\frakm_{p_3}^4} & \tilde \Psi_{30} &\equiv \overline A_2 +\overline A_3 \pmod{\frakm_{\overline p_3}^4}\\
\tilde \Psi_{24} &\equiv B_2+B_3 \pmod{\frakm_{p_3}^4} & \tilde \Psi_{24} &\equiv \overline B_2 +\overline B_3 \pmod{\frakm_{\overline p_3}^4}\\
\tilde \Psi_{12} &\equiv C_0+C_1 \pmod{\frakm_{p_3}^2} & \tilde \Psi_{12} &\equiv \overline C_0 +\overline C_1 \pmod{\frakm_{\overline p_3}^2}\\
\tilde \Psi_{6} &\equiv D_0+D_1 \pmod{\frakm_{p_3}^2} & \tilde \Psi_{6} &\equiv \overline D_0 +\overline D_1 \pmod{\frakm_{\overline p_3}^2}.
\end{align*}
There are constants $\mu,\nu,\overline \mu,\overline \nu\in \C$ such that
\begin{align*}
A_2 &= \mu B_2 &\overline A_2 &= \overline \mu \overline B_2\\
C_0 &= \nu D_0 &\overline C_0 &= \overline \nu \overline D_0.
\end{align*}
Since the invariant $D_0^2\Psi_{12}-C_0\Psi_6^2$ vanishes at $p_3$, it is double at $p_3$, and thus $$0 \equiv D_0^2 \tilde \Psi_{12}-C_0\tilde \Psi_6^2\equiv D_0^2C_1 -2C_0D_0D_1 \equiv D_0^2(C_1-2\nu D_1)\pmod {\frakm_{p_3}^2},$$ so that \begin{align*}C_1 &= 2\nu D_1\\ \overline C_1 &= 2\overline \nu \overline D_1.\end{align*}

We can also expand $\tilde\Psi_d$ in local coordinates around $p_4$; this turns out to be considerably simpler.  Observe that $[0:0:1]$ is one of the quadruple points of the configuration, so there is no need to change coordinates to express an invariant $\Psi_d$ in local coordinates at $p_4$.  The presence of the transformations $R_1$ and $R_2$ in the group $\tilde G$ imply that if $x^ay^bz^c$ is a monomial which appears in a $\tilde G$-invariant homogeneous form $\Psi_d$ then $a,b,c$ have the same parity.  If $d$ is divisible by $6$, then the exponents must all be even.  It follows that the functions $\tilde \Psi_d$ have expansions \begin{align*}\tilde \Psi_{30} &\equiv \widehat A_2 \pmod {\frakm_{p_4}^4}\\
\tilde \Psi_{24} &\equiv \widehat B_0 {\pmod {\frakm_{p_4}^2}} \\
\tilde \Psi_{12} & \equiv \widehat C_2 {\pmod {\frakm_{p_4}^4}}\\
\tilde \Psi_{6} & \equiv \widehat D_0 {\pmod {\frakm_{p_4}^2}}.
\end{align*}
By Corollary \ref{cor-parallel} there are also constants $\widehat \mu,\widehat \nu\in \C$ defined so that \begin{align*}
\widehat A_2 &= \widehat\mu\widehat C_2\\ \widehat B_0&= \widehat\nu\widehat D_0.
\end{align*}

\subsection{Equation of the curve of class $90H-4E_4-8E_3$}For constants $\lambda_i\in \C$, we study the curve $\Psi_{90}=0$ with equation $$\Psi_{90}:=\lambda_1 \Psi_{30}^3+\lambda_2 \Psi_6\Psi_{24}\Psi_{30}^2+\lambda_3 \Psi_6^2\Psi_{24}^2\Psi_{30}+\lambda_4\Psi_{12}\Psi_{24}^2\Psi_{30}+\lambda_5\Psi_6\Psi_{12}\Psi_{24}^3=0,$$ and seek to determine values for the $\lambda_i$ that will make
$\Psi_{90}=0$ be the curve $B$ that we are looking for.  The main difference with the Klein case is that we now have $3$ orbits of points at which to assign multiplicities.  We begin with the point $p_4$ since it is the most different from the Klein.

\begin{lemma}\label{lem-Wiman4uple}
The curve $\Psi_{90}$ is $4$-uple at $p_4$ if and only if $$\widehat\mu\lambda_3  +\widehat\nu \lambda_5 = 0.$$ Intrinsically, the curve is $4$-uple at $p_4$ if and only if $$\lambda_3 + \left(\frac{\Psi_{12}\Psi_{24}}{\Psi_6\Psi_{30}}\right)(p_4) \cdot \lambda_5=0.$$
\end{lemma}
\begin{proof}
We expand $\tilde \Psi_{90}$ at $p_4$, working mod $\frakm_p^4$.  We have \begin{align*}\tilde\Psi_{90} &= \lambda_1 \tilde\Psi_{30}^3+\lambda_2 \tilde\Psi_6\tilde\Psi_{24}\tilde\Psi_{30}^2+\lambda_3 \tilde\Psi_6^2\tilde\Psi_{24}^2\tilde\Psi_{30}+\lambda_4\tilde\Psi_{12}\tilde\Psi_{24}^2\tilde\Psi_{30}+\lambda_5\tilde\Psi_6\tilde\Psi_{12}\tilde\Psi_{24}^3\\
&= \lambda_3 \widehat D_0^2 \widehat B_0^2 \widehat A_2+\lambda_5 \widehat D_0\widehat C_{2} \widehat B_0^3\\
&= (\lambda_3 \widehat \mu \widehat\nu^2+\lambda_5 \widehat\nu^3)\widehat D_0^4\widehat C_2,
\end{align*}
from which the result is immediate.
\end{proof}

Next we consider the requirement for $\Psi_{90}$ to be $7$-uple at one of the triple points.

\begin{lemma}\label{lem-Wiman7uple}
The curve $\Psi_{90}$ is $7$-uple at $p_3$ if and only if
$$\mu^3\lambda_1+\mu^2D_0 \lambda_2+\mu D_0^2 \lambda_3 + \mu \nu D_0 \lambda_4+ \nu D_0^2 \lambda_5 =0.$$
\end{lemma}
\begin{proof}
Expand $\tilde \Psi_{90}$ at $p_3$, working mod $\frakm_{p_3}^7$.  We get
\begin{align*}\tilde\Psi_{90} &= \lambda_1 \tilde\Psi_{30}^3+\lambda_2 \tilde\Psi_6\tilde\Psi_{24}\tilde\Psi_{30}^2+\lambda_3 \tilde\Psi_6^2\tilde\Psi_{24}^2\tilde\Psi_{30}+\lambda_4\tilde\Psi_{12}\tilde\Psi_{24}^2\tilde\Psi_{30}+\lambda_5\tilde\Psi_6\tilde\Psi_{12}\tilde\Psi_{24}^3\\
&=\lambda_1 A_2^3+\lambda_2 D_0B_2A_2^2+\lambda_3 D_0^2B_2^2A_2+\lambda_4 C_0B_2^2A_2+\lambda_5 D_0C_0B_2^3\\
&= \lambda_1 \mu^3 B_2^3+\lambda_2 \mu^2D_0 B_2^3+\lambda_3 \mu D_0^2 B_2^3+\lambda_4\mu\nu D_0 B_2^3 + \lambda_5 \nu D_0^2 B_2^3\\
&= (\lambda_1\mu^3 + \lambda_2 \mu^2D_0 + \lambda_3 \mu D_0^2 + \lambda_4 \mu \nu D_0 + \lambda_5 \nu D_0^2)B_2^3,
\end{align*}
which proves the result.
\end{proof}

Next we analyze the further condition which gives that $\Psi_{90}$ is $8$-uple at a triple point.

\begin{lemma}\label{lem-Wiman8uple}
The curve $\Psi_{90}$ is $8$-uple at $p_3$ if  it is $7$-uple at $p_3$ and
$$3\mu^2\lambda_1+ 2\mu D_0\lambda_2 + D_0^2\lambda_3 + \nu D_0\lambda_4 =0.$$
Intrinsically, the curve is $8$-uple at $p_3$ whenever the $\lambda_i$ satisfy the system
\begin{align*}
3\cdot\lambda_1+2\left(\frac{\Psi_6\Psi_{24}}{\Psi_{30}}\right)(p_3) \cdot \lambda_2 + \left(\frac{\Psi_6^2\Psi_{24}^2}{\Psi_{30}^2}\right)(p_3)\cdot \lambda_3+\left(\frac{\Psi_{12}\Psi_{24}^2}{\Psi_{30}^2}\right)(p_3)\cdot \lambda_4 &=0\\
\lambda_2 + 2 \left(\frac{\Psi_6\Psi_{24}}{\Psi_{30}}\right)(p_3)\cdot \lambda_3+2 \left(\frac{\Psi_{12}\Psi_{24}}{\Psi_6\Psi_{30}}\right)(p_3)\cdot \lambda_4 +3\left(\frac{\Psi_{12}\Psi_{24}^2}{\Psi_{30}^2}\right)(p_3)\cdot \lambda_5 &= 0.
\end{align*}
\end{lemma}
\begin{proof}
The proof is highly similar to the proof of Lemma \ref{lem-8upleKlein}, so we omit it.\end{proof}

In total, we have found the following criterion for there to be a curve $\Psi_{90}=0$ with the required multiplicities.

\begin{proposition}\label{prop-WimanMatrix}
The curve $\Psi_{90}=0$ is $4$-uple at $p_4$ and $8$-uple at both $p_3$ and $\overline p_3$ if the $\lambda_i$ are a solution of the system
\renewcommand{\arraystretch}{2}
$$\begin{pmatrix}
3 & 2 \left(\frac{\Psi_6\Psi_{24}}{\Psi_{30}}\right)(p_3) & \left(\frac{\Psi_6^2\Psi_{24}^2}{\Psi_{30}^2}\right)(p_3) & \left(\frac{\Psi_{12}\Psi_{24}^2}{\Psi_{30}^2}\right)(p_3) & 0\\
0 & 1 & 2\left(\frac{\Psi_6\Psi_{24}}{\Psi_{30}}\right)(p_3) & 2\left(\frac{\Psi_{12}\Psi_{24}}{\Psi_6\Psi_{30}}\right)(p_3) & 3\left(\frac{\Psi_{12}\Psi_{24}^2}{\Psi_{30}^2}\right)(p_3)\\
3 & 2 \left(\frac{\Psi_6\Psi_{24}}{\Psi_{30}}\right)(\overline p_3) & \left(\frac{\Psi_6^2\Psi_{24}^2}{\Psi_{30}^2}\right)(\overline p_3) & \left(\frac{\Psi_{12}\Psi_{24}^2}{\Psi_{30}^2}\right)(\overline p_3) & 0\\
0 & 1 & 2\left(\frac{\Psi_6\Psi_{24}}{\Psi_{30}}\right)(\overline p_3) & 2\left(\frac{\Psi_{12}\Psi_{24}}{\Psi_6\Psi_{30}}\right)(\overline p_3) & 3\left(\frac{\Psi_{12}\Psi_{24}^2}{\Psi_{30}^2}\right)(\overline p_3)\\
0 & 0 & 1 & 0 &  \left(\frac{\Psi_{12}\Psi_{24}}{\Psi_6\Psi_{30}}\right)(p_4)
\end{pmatrix} \begin{pmatrix}\lambda_1 \vphantom{\left(\frac{\Psi_6^2}{\Psi_{30}^2}\right)}\\ \lambda_2 \vphantom{\left(\frac{\Psi_6^2}{\Psi_{30}^2}\right)}\\ \lambda_3\vphantom{\left(\frac{\Psi_6^2}{\Psi_{30}^2}\right)} \\ \lambda_4 \vphantom{\left(\frac{\Psi_6^2}{\Psi_{30}^2}\right)}\\ \lambda_5 \vphantom{\left(\frac{\Psi_6^2}{\Psi_{30}^2}\right)}\end{pmatrix}=0.$$ \renewcommand{\arraystretch}{1}
\end{proposition}
\begin{proof}
This follows immediately from Lemmas \ref{lem-Wiman4uple} and \ref{lem-Wiman8uple}, noting that the obvious analog of Lemma \ref{lem-Wiman8uple} holds for the triple point $\overline p_3$.
\end{proof}

Unfortunately we are not aware of a simple reason why the matrix in the proposition actually has rank $4$ instead of $5$; this is why the existence of the curve $B$ is surprising.  We now fix scalar multiples of the invariants $\Psi_d$ to explicitly compute the entries in the above matrix.  We choose invariants
\begin{align*}
\Psi_6 &= 2\Phi_6\\
\Psi_{12} &= 18 (\Phi_6^2-\Phi_{12})\\
\Upsilon_{12} &= \Psi_6^2-\frac{1}{60}(15+s)\Psi_{12} \qquad \qquad \qquad \qquad s=i\sqrt{15}\\
\overline\Upsilon_{12} &= \Psi_6^2-\frac{1}{60}(15-s)\Psi_{12}.\\
\Psi_{24} &= \Upsilon_{12}\overline\Upsilon_{12} = \Psi_6^4 - \frac{1}{2}\Psi_6^2\Psi_{12}+\frac{1}{15}\Psi_{12}^2.\\
\Psi_{30} &= \frac{36}{25}(2\Phi_6^5-11\Phi_6^3\Phi_{12}+36\Phi_6\Phi_{12}^2-27\Phi_{30})
\end{align*}
The auxiliary invariants $\Upsilon_{12},\overline\Upsilon_{12}$ are specified up to scale by the requirement that they pass through $p_3$ and $\overline p_3$, respectively.  While they are not defined over $\Q$ in terms of $\Psi_6$ and $\Psi_{12}$, the invariant $\Psi_{24} = \Upsilon_{12}\overline\Upsilon_{12}$ is defined over $\Q$ in terms of $\Psi_6$ and $\Psi_{12}$.

\begin{corollary}\label{cor-WimanEqn}
If the invariants $\Psi_d$ are normalized as above, then the curve $$\Psi_{90}:=4 \Psi_{30}^3-10\Psi_6\Psi_{24}\Psi_{30}^2-20\Psi_6^2\Psi_{24}^2\Psi_{30}+10\Psi_{12}\Psi_{24}^2\Psi_{30}-5\Psi_6\Psi_{12}\Psi_{24}^3=0$$ is $4$-uple at $p_4$ and $8$-uple at each of $p_3,\overline p_3$.
\end{corollary}
\begin{proof}
We compute the entries of the matrix of Proposition \ref{prop-WimanMatrix}, scaling the rows to clear denominators.  Putting $s=i\sqrt{15}$, the matrix becomes
$$\begin{pmatrix}
30 & 10 + 2s & 1+s & 4s & 0\\
0 & 5 & 5+s & 15+5s & 6s\\
30 & 10 - 2s & 1-s & -4s & 0\\
0 & 5 & 5-s & 15-5s & -6s\\0 & 0 & 1 & 0 & -4
\end{pmatrix}.$$ The vector $(4,-10,-20,10,-5)^T$ is evidently in the kernel.
\end{proof}

\subsection{$G$-irreducibility of the curve of class $90H - 4E_4 - 8E_3$}

With the equation of the curve $B$ in hand, we now complete the proof of the main theorem in this section.  


\begin{proof}[Proof of Theorem \ref{thm-negativeCurveWiman}]

By Corollary \ref{cor-WimanEqn}, if we normalize the invariants $\Psi_d$ appropriately then the curve $B$ in $\P^2$ defined by the equation
$$4 \Psi_{30}^3-10\Psi_6\Psi_{24}\Psi_{30}^2-10(2\Psi_6^2-\Psi_{12})\Psi_{24}^2\Psi_{30}-5\Psi_6\Psi_{12}\Psi_{24}^3=0$$ has the required multiplicities.  Everything except the $G$-irreducibility of this curve follows exactly as for the Klein configuration, so we focus on $G$-irreducibility.

Working in the weighted projective space $\P(6,12,30)$ with coordinates $w_0,w_1,w_2$, we define forms \begin{align*}
L &= w_0^2-\frac{1}{60}(15+s) w_1\\
\overline L &= w_0^2 - \frac{1}{60}(15-s)w_1\\
Q &= L\overline L\\
F &= 4w_2^3-10w_0Qw_2^2-10(2w_0^2-w_1)Q^2w_2-5w_0w_1Q^3.
\end{align*}
As in the proof of Theorem \ref{thm-negativeCurveKlein}, it is enough to show  that $F$  is irreducible in $\C[w_0,w_1,w_2]$.

First suppose that $F$ factors into irreducible factors as $$F = (F_1 w_2+F_0)(G_1 w_2 + G_0)(H_1 w_2 + G_0),$$ where $F_i,G_i,H_i\in \C[w_0,w_1]$ and the factors are weighted homogeneous of degree $30$.  Then $F_0,G_0,H_0$ are weighted homogeneous of degree $30$.  We have $$F_0G_0H_0 = -5 w_0w_1L^3\overline L^3,$$ and the right hand side is already factored into irreducibles.  Since $w_1,L,\overline L$ each have degree $12$, this is absurd.

Next suppose that $F$ factors into irreducible factors as $$F = (G_2w_2^2+G_1w_2+G_0)(H_1w_2+H_0)=:GH$$ where $G_i,H_i\in \C[w_0,w_1]$ are weighted homogeneous of the appropriate degrees. Eliminating the possibility that $F$ factors in this way is more delicate; for instance, it depends on the particular numerical coefficients in the definition of $F$.

Since $F$ is defined over $\Q$ and its two irreducible factors have different degrees, if $\sigma$ is a field automorphism of $\C$ then the action of $\sigma$ on $\P(6,12,30)$ fixes the curves $G=0$ and $H=0$. This implies that there is some nonzero $\lambda\in \C$ such that all the coefficients of $G$ (resp. $H$) are rational multiples of $\lambda$ (resp. $\lambda^{-1}$).  Eliminating $\lambda$, we may as well assume $G,H$ have $\Q$-coefficients.

Let us compare coefficients of $F$ and $GH$ to determine the irreducible factors of the $G_i,H_i$.  We write $\sim$ to denote an equality which holds up to a scalar multiple.  First observe
$$G_2\sim 1  \qquad \deg G_1=30 \qquad \deg G_0 = 60 \qquad H_1 \sim 1 \qquad \deg H_0 = 30.$$ Examining the coefficient of $w_2^0$ gives $$G_0H_0 \sim w_0w_1Q^3 = w_0w_1L^3\overline L^3.$$ Since $G_0,H_0$ have $\Q$-coefficients, the only possibility is that \begin{align*}
G_0 &\sim w_1Q^2\\
H_0&\sim w_0Q.
\end{align*}
Next, looking at the coefficient of $w_2^2$, $$G_1H_1+G_2H_0\sim w_0Q,$$ from which it follows that $$G_1 \sim w_0Q.$$ Note that the relation $G_1H_0+G_0H_1\sim (2w_0^2-w_1)Q^2$ is consistent with the factorizations of the $G_i,H_i$ that we have found so far, so to go further we must consider the numerical coefficients of the factors.

Let $g_i,h_i\in \C$ be such that $$G_2 = g_2 \qquad G_1 = g_1 w_0Q \qquad G_0 = g_0w_1Q^2 \qquad H_1=h_1 \qquad H_0=h_0w_0Q.$$
Comparing coefficients gives relations
\begin{align*}
g_2 h_1 &= 4\\
g_2 h_0 + g_1h_1 &= -10\\
g_1 h_0 &= -20\\
g_0 h_1 &= 10\\
g_0h_0 &= -5.
\end{align*}
However, this system has no solutions.  Indeed, the identity
$$(g_0h_1)(g_0h_0)(g_2h_0+g_1h_1)=(g_2h_1)(g_0h_0)^2+(g_1h_0)(g_0h_1)^2$$ shows that if all the equations in the system except the second are satisfied then $$g_2h_0+g_1h_1 = \frac{4\cdot(-5)^2+(-20)\cdot 10^2}{10\cdot (-5)} = 38.$$  We conclude that $F$ is irreducible.
\end{proof}

\section{Generators and asymptotic resurgence}\label{sec-asympRes}

We can now use our results on Waldschmidt constants to compute the asymptotic resurgence of the Wiman configuration and bound the asymptotic resurgence of the Klein configuration.  The main additional information we need is knowledge of the generators of the ideal $I_\cL$.

\subsection{Jacobians and invariant ideals}  In this subsection we prove a basic fact about the relationship between Jacobian ideals and group actions.  For this subsection only, we let $S=K[x_0,\ldots,x_n]$ be a polynomial ring over a field $K$ and suppose $G$ is a group that acts linearly on $S$. For a polynomial $f\in S$ we write $$\nabla f=\left[\partial f/\partial x_0\ \ldots \ \partial f/\partial x_n \right]^T$$ for the gradient vector of partial derivatives of $f$. We will need the following identity.

\begin{lemma}\label{lem-chainrule}
Suppose $g\in G$ acts on $S$ via the matrix $A_g\in \GL_{n+1}(K)$.  For any $f\in S$ we have $$g(\nabla f) =A_g^{-1}\cdot \nabla g(f).$$
\end{lemma}

The proof is a straightforward application of the multivariable chain rule, so we omit it.

\begin{lemma}\label{lem-jacinvariant}
Suppose $f_1,\ldots,f_s\in S^G$ are $G$-invariant.  Let $$ J = [\nabla f_1 \ \ldots\  \nabla f_s]$$ be the $(n+1)\times s$ matrix with columns given by the gradient vectors of the $f_i$.  If $I_J$ is the ideal of maximal minors of $J$, then $I_J$ is $G$-invariant.
\end{lemma}
\begin{proof}
Let $g\in G$ and use Lemma \ref{lem-chainrule} to compute the action of $g$ on  $J$ as follows:
$$g(J)=g\left[ \nabla f_1 \ldots \nabla f_s \right ]=\left[ A_g^{-1}\nabla g(f_1) \ \ldots \ A_g^{-1}\nabla g(f_s) \right ]=A_g^{-1}\left[ \nabla f_1 \ldots  \nabla f_s \right ]=A_g^{-1}J.$$
In the displayed equation above, the penultimate equality uses that $f_1,\ldots, f_s$ are $G$-invariant. The identity $g(J)=A_g^{-1}J$ implies that the ideals of maximal minors for $J$ and $g(J)$ are the same, i.e. $I_{g(J)}=I_{J}$. Since the action of $G$ respects the multiplicative structure of $S$, in particular taking minors to minors, we also have that  $g(I_{J})=I_{g(J)}$. We conclude $g(I_J) = I_J$.
\end{proof}

\subsection{Generators of ideals} Lemma \ref{lem-jacinvariant} allows us to identify the ideals $I_\cL$ of the Klein and Wiman configurations as natural ideals arising from the fundamental invariant forms.  We begin with the Klein case.

\begin{proposition}\label{prop-gensKlein}
The homogeneous ideal $I_\cK$ of the $49$ points in the Klein configuration is the ideal of $2\times 2$ minors of the matrix $$J = \begin{pmatrix}
\partial \Phi_4/\partial x & \partial \Phi_4/\partial y & \partial \Phi_4/\partial z \\
\partial \Phi_6/\partial x & \partial \Phi_6/\partial y & \partial \Phi_6/\partial z \\
\end{pmatrix},$$ where $\Phi_4,\Phi_6$ are the invariants of \S \ref{ssec-KleinInvariantsPrelim}. In particular, $\alpha(I_\cK) = \omega(I_\cK) = 8$ and $I_\cK$ is minimally generated by $3$ generators of degree $8$.
\end{proposition}
Note that a different proof of the follow-up statements was given in \cite[Proposition 4.2]{S14}.
\begin{proof}
Let $I$ be the ideal of $2\times 2$ minors of $J$; we prove $I = I_{\cK}$.  Since the Klein quartic $\Phi_4=0$ is smooth, the entries of the gradient vector $\nabla \Phi_4$ generate an ideal primary to the maximal ideal of $S = \C[x,y,z]$.  Thus they form a regular sequence.  By \cite{refEH} or \cite{refTo}, the occurrence of a syzygy on the generators of $I$ given by a regular sequence implies that  the quotient $S/I$ is Cohen-Macaulay with Hilbert-Burch matrix $J^T$ and minimal free resolution
$$0\rightarrow S(-13)\oplus S(-11)\rightarrow S(-8)^3\rightarrow S\rightarrow S/I\rightarrow 0.$$
 In particular, $I$ is saturated and $S/I$ is the coordinate ring of a set of (not necessarily reduced) points in $\P^2$. Furthermore, the above free resolution of $S/I$ allows us to compute $\deg S/I = 49$.

Since $I$ is $G_\cK$-invariant by Lemma \ref{lem-jacinvariant}, the support of $S/I$ is a union of orbits of the $G$-action on $\P^2$.  Additionally, $S/I$ has the same length at each point of an orbit.  By Remark \ref{rem-KleinOrbit}, we see that there are nonnegative integers $a_i$ such that $$28a_1+21a_2+24a_3+56a_4+42a_5+84a_6+168a_7 = 49.$$ The only solution in nonnegative integers to this equation is visibly $a_1=a_2=1$ and $a_i = 0$ ($i\geq 3$), corresponding to $I=I_{\cK}$ being the ideal of the triple and quadruple points of $\cK$.
\end{proof}

A similar approach works for the Wiman configuration.

\begin{proposition}\label{prop-gensWiman}
The homogeneous ideal $I_{\cW}$ of the $201$ points in the Wiman configuration is the ideal of $2\times 2$ minors of the matrix $$J = \begin{pmatrix}
\partial \Phi_6/\partial x & \partial \Phi_6/\partial y & \partial \Phi_6/\partial z \\
\partial \Phi_{12}/\partial x & \partial \Phi_{12}/\partial y & \partial \Phi_{12}/\partial z \\
\end{pmatrix},$$ where $\Phi_6,\Phi_{12}$ are the invariants of \S\ref{ssec-WimanInvariants}. In particular, $\alpha(I_\cW) = \omega(I_\cW) = 16$ and $I_\cW$ is minimally generated by $3$ generators of degree $16$.
\end{proposition}
\begin{proof}
Let $I$ be the ideal of $2\times 2$ minors of $J$, so that $I$ is $G_\cW$-invariant by Lemma \ref{lem-jacinvariant}.  Since the Wiman sextic $\Phi_6=0$ is smooth, the same argument as in the proof of Proposition \ref{prop-gensKlein}, shows that $S/I$ has minimal free resolution $$0\to S(-27)\oplus S(-21) \to S(-16)^3 \to S \to S/I \to 0.$$ The resolution implies that $\deg S/I = 201$.    As in the Klein case, by Remark \ref{rem-WimanOrbits} this yields a solution in nonnegative integers to the equation$$60a_1 + 45a_2 + 36a_3 + 72a_4 + 90a_5 + 180a_6 + 360 a_7 = 201.$$ It is easy to see that the only solution to this equation in nonnegative integers has $a_1 = 2$, $a_2=a_3=1$, and $a_i = 0$ ($i\geq 4$).

This leaves two possibilities: either $I=I_\cW$, or $S/I$ has length $2$ at all of the points in one of the orbits of triple points.  In the latter case, we find that there is a length $2$ scheme supported at a triple point $p$ of the configuration $\cW$ which is invariant under $G_p\cong D_6$.  Then the tangent direction spanned by this scheme gives a $G_p$-invariant subspace of the tangent space $T_p \P^2$, contradicting Lemma \ref{lem-leadingTermConfig} (1).  Therefore $I=I_\cW$.
\end{proof}

\subsection{Asymptotic resurgence} Our results on Waldschmidt constants and our knowledge of $\alpha(I_{\cL})$ and $\omega(I_{\cL})$ now provide estimates on the asymptotic resurgence of $I_{\cK}$ and allow us to compute the asymptotic resurgence of $I_{\cW}$ exactly.

\begin{theorem}
For the Klein configuration of lines, we have
$$1.230 \approx \frac{16}{13} \leq  \widehat \rho(I_{\cK}) \leq \frac{816}{661} \approx 1.234.$$ For the Wiman configuration of lines,  $$\widehat\rho(I_{\cW}) = \frac{32}{27} \approx 1.185.$$
\end{theorem}
\begin{proof}
Recall that for any ideal $I$ we have $$ \frac{\alpha(I)}{\widehat\alpha(I)} \leq \widehat\rho (I) \leq \frac{\omega(I)}{\widehat \alpha(I)}.$$ Since  $\alpha(I_\cL)=\omega(I_\cL)$ for $\cL = \cK$ or $\cW$ by Propositions \ref{prop-gensKlein} and \ref{prop-gensWiman}, the result  follows from Theorem \ref{thm-D7} and Corollary \ref{cor-WimanWaldschmidt}.
\end{proof}

\begin{remark}
For the Klein configuration, the weaker upper bound $$\widehat\rho(I_{\cK})< \frac{36}{29} \approx 1.242$$ follows from Corollary \ref{cor-KleinLowerThy}, which did not require computer calculations.  Conjecture \ref{conj-Klein} would imply that in fact $\widehat \rho(I_\cK) = 16/13$.
\end{remark}

\section{Failure of containment and resurgence}\label{sec-resurgence}

The resurgences of the Klein and Wiman configurations can be computed exactly.  We begin with the failure of containment that achieves the supremum in the definition of resurgence.  In the case of the Klein configuration a computer-free but computationally heavy proof of the next result was first given in \cite{S14}.  We offer two new proofs here that use tools from representation theory.

\begin{proposition}\label{prop-containmentFailure}
If $I_{\cL}$ is the ideal of the Klein or Wiman configurations of points, then there is a failure of containment $I_{\cL}^{(3)}\not\subseteq I_{\cL}^2$.  More precisely, the product of the linear forms defining the configuration is an element of $I_{\cL}^{(3)}$ which is not in $I_{\cL}^2$.
\end{proposition}

The fact that the product of the lines is contained in $I_{\cL}^{(3)}$ is clear since both configurations only have points of multiplicity $3$ or higher.  Our first proof makes use of the character theory of the group.

\begin{proof}[First proof]
We begin with the Klein configuration.  We claim that there are no invariant forms in the degree $21$ piece $(I^2_{\cK})_{21}$.  Note that $(I^2_{\cK})_{21}$ is a finite-dimensional representation of $G_\cK$.  Letting $S$ be the homogeneous coordinate ring, the multiplication map $$(I^2_{\cK})_{16} \otimes S_5 \to (I_{\cK}^2)_{21}$$ is a surjective map of $G$-modules.  Thus every irreducible submodule of $(I_{\cK}^2)_{21}$ appears in $(I_{\cK}^2)_{16}\otimes S_5$ by Schur's Lemma, and in particular the induced map on $G$-invariants is surjective.  Thus, to prove the claim, it will be enough to show that $(I_{\cK}^2)_{16}\otimes S_5$ has no trivial submodules.

Let $V=S_1^*$ be the $3$-dimensional irreducible representation of $G$ which gives rise to the Klein configuration.  From the character table of $G$ (see \cite{Elkies}) we know that $V$ and $V^*$ are the only $3$-dimensional irreducible representations of $G$ and the only $1$-dimensional representation of $G$ is the trivial representation.  Since $(I_{\cK})_8$ is $3$-dimensional and contains no invariants ($\Phi_4$ is not in $I_{\cK}$), we deduce that it is isomorphic to either $V$ or $V^*$.  Both $V$ and $V^*$ have the same symmetric square $\Sym^2 V \cong \Sym^2 V^*$, which is the unique irreducible $6$-dimensional representation of $G$.  Then the natural map $\Sym^2(I_{\cK})_8 \to (I_{\cK}^2)_{16}$ is a nonzero surjective map of $G$-modules since $I_\cK$ is generated in degree $8$ by Proposition \ref{prop-gensKlein}, so it is an isomorphism by Schur's Lemma.  Since $S_5 \cong \Sym^5 V^*$, our question is to determine whether $$\Sym^2 V\otimes \Sym^5 V^*$$ contains a trivial submodule.  This can be determined immediately from the character of this representation, which we now compute.

First we recall the character of $V^*$ and $\Sym^2 V$, as well as the conjugacy class data for $G$.  We also display some values for the character of $\Sym^5 V^*$ which we will derive in a moment.  Blank entries in $\chi_{\Sym^5 V^*}$ will not be needed in our computation.  The conjugacy classes are labeled by the order of an element and a letter to distinguish between several classes consisting of elements of the same order.  For example, class 7A is one of two classes consisting of elements of order $7$.
$$
\begin{array}{c|cccccc}
c & 1A & 2A & 3A & 4A & 7A & 7B\\
\#c & 1 & 21 & 56 & 42 & 24 & 24\\
\hline
\chi_{V^*} & 3 & -1 & 0 & 1 & \overline\alpha & \alpha\\
\chi_{\Sym^2 V} & 6 & 2 & 0 & 0 & -1 & -1\\
\chi_{\Sym^5 V^*} & 21 & -3 &  &  & 0 & 0
\end{array} \qquad \begin{array}{rcl}\alpha&=&\zeta+\zeta^2+\zeta^4 \\ \zeta^7 &=& 1\end{array}
$$
Observe that the indicated entries for $\chi_{\Sym^5 V}$ are enough to prove the theorem.  Indeed, $\chi_{\Sym^2 V\otimes \Sym^5V^*}$ takes value $6\cdot 21$ on class $1A$, value $2\cdot (-3)$ on class $2A$, and $0$ on all other conjugacy classes.  Its inner product with the trivial character is then $0$, so there are no trivial submodules in $\Sym^2 V\otimes \Sym^5 V^*$.

To compute the indicated values for $\chi_{\Sym^5 V^*}$, we first recall how to compute the character.  Suppose the action of the group element $g\in G$ on $V^*$ has eigenvalues $\lambda_1,\lambda_2,\lambda_3$.  Let $p(x,y,z)$ be the sum of all monomials in $x,y,z$ of degree $5$.  Then $$\chi_{\Sym^5 V^*}(g) = p(\lambda_1,\lambda_2,\lambda_3).$$
The given entries in the character table now follow from easy combinatorics, as follows.

 To compute the character on the class $2A$, observe that such a group element $g$ acts on $V^*$ with eigenvalues $1,-1,-1$.   The number of monomials $x^ay^bz^c$ of degree $5$ such that $b\equiv c \pmod 2$ is $9$, while there are $12$ monomials with $b\not\equiv c \pmod 2$.  Thus $p(1,-1,-1) = -3$.

For the class $7B$, there is a group element $g$ acting on $V^*$ with eigenvalues $\zeta,\zeta^2,\zeta^4$.  If we weight the variables $x,y,z$ with $\Z/7\Z$ degrees $1,2,4$ and partition the monomials of (ordinary) degree $5$ according to their $\Z/7\Z$-degree, we find there are precisely $3$ monomials of each $\Z/7\Z$-degree.  Thus the fact that $\chi_{\Sym^5 V^*}(g) = 0$ follows from the identity $$1+\zeta+\zeta^2+\zeta^3+\zeta^4+\zeta^5+\zeta^6 = 0.$$ The value on class $7A$ must be conjugate to the value on class $7B$, so is also $0$.

The argument for the Wiman configuration follows an identical outline, although at first glance the character table is more intimidating (the full character table can be obtained in GAP by the command \texttt{CharacterTable("3.A6")}, but we will only need a very small portion of it here).  In the end, however, the amount of computation we must do is the same as for the Klein.  We show there are no invariants in $(I_{\cW}^2)_{45}$ by showing that $(I_\cW^2)_{32}\otimes S_{13}$ has no trivial submodule.

Let $V = S_1^*$ be the $3$-dimensional representation of the Valentiner group $ \tilde G = \tilde G_\cW$ which gives rise to the Wiman configuration.  Again $V'=(I_{\cW})_{16}$ is a $3$-dimensional irreducible representation, and its symmetric square $\Sym^2 V'$ is isomorphic to $(I_{\cW}^2)_{32}$ and is an irreducible $6$-dimensional representation.  The group $\tilde G$ has $4$ different $3$-dimensional irreducible characters and $2$ different $6$-dimensional irreducible characters; we display one of each $\chi_3,\chi_6$ below, choosing $\chi_6$ to be the character of the symmetric square of the representation corresponding to $\chi_3$.    The alternate characters are related by complex conjugation and/or an automorphism exchanging $\pm\sqrt{5}$; our argument will not be sensitive to this.  We let $\psi$ be the character corresponding to the $13$th symmetric power of the representation corresponding to $\chi_3$, and display some of its values which we will verify. We group the conjugacy classes in a way that emphasizes the fact that $\tilde G$ is the triple cover $3\cdot A_6$.

\begin{align*}&\begin{array}{c|ccc:ccc:cc:ccc}
c & 1A & 3A & 3B & 2A & 6A & 6B & 3C & 3D & 4A & 12A & 12B \\
\#c & 1 & 1 & 1 & 45 & 45 & 45 & 120 & 120 & 90 & 90 & 90 \\
\hline
\chi_3 & 3 & 3\omega & 3\omega^2 & -1 & -\omega & -\omega^2 & 0 & 0 & 1 & \omega & \omega^2 \\
\chi_6 & 6 & 6\omega^2 & 6\omega & 2 & 2\omega^2 & 2\omega & 0 & 0 &0 & 0& 0 \\ \psi & 105 & 105\omega & 105\omega^2 & -7 & -7\omega & -7\omega^2 &&&&&
\end{array}\\\\
&\begin{array}{c|ccc:ccc}
c&5A&15A&15B&5B&15C&15D\\
\#c & 72 &72&72&72&72&72\\
\hline \chi_3 & -\mu_1 & -\mu_1\omega & -\mu_1\omega^2 & -\mu_2 & -\mu_2 \omega & -\mu_2 \omega^2\\
\chi_6 & 1 & \omega^2 & \omega & 1 & \omega^2 & \omega\\
\psi & 0&0&0&0&0&0
\end{array} \qquad \begin{array}{rcl} \omega^3 &=& 1\\ \mu_1 &=& \frac{-1+\sqrt{5}}{2}\\ \mu_2 & = & \frac{-1-\sqrt{5}}{2}\end{array}
\end{align*}

As with the Klein, observe that if we establish the displayed values for $\psi$ then both $\chi_6 \otimes \psi$ and $\overline\chi_6 \otimes \psi$ are orthogonal to the trivial character; the same result also clearly holds if we define $\psi$ in terms of any of the other conjugate $3$-dimensional characters.  Thus, whichever irreducible $6$-dimensional representation $(I_{\cW}^2)_{32}$ is, the representation $(I_{\cW}^2)_{32} \otimes S_{13}$ has no trivial submodule.

To compute the displayed values of $\psi$, it is enough to compute the values on classes $2A$ and $5A$.  This is because the center of $\tilde G$ acts on the conjugacy classes by permuting the blocks of $3$ columns.  Furthermore, since the values of $\chi_3$ on classes $5A$ and $5B$ are conjugate under the automorphism exchanging $\pm\sqrt{5}$, the same holds for $\psi$.

The value of $\psi$ on $2A$ follows from the same logic as in the Klein case.  An element of class $2A$ has eigenvalues $1,-1,-1$.  There are $49$ monomials $x^ay^bz^c$ of degree $13$ with $b\equiv c \pmod 2$, and $56$ with $b\not\equiv c\pmod 2$.  Thus the value on $2A$ is $-7$.

For the value of $\psi$ on $5A$, we have $-\mu_1 = 1+\zeta^2 + \zeta^3$, where $\zeta = e^{2\pi i/5}$.  An element of class $5A$ has eigenvalues $1,\zeta^2,\zeta^3$.  Give the degree $13$ monomial $x^ay^bz^c$ a $\Z/5\Z$ degree of $0a+2b+3c$.  Partitioning the monomials of degree 13 by their $\Z/5\Z$-degree, there are $21$ monomials in each class.  Since $$1+\zeta+\zeta^2+\zeta^3+\zeta^4 = 0,$$ we conclude that the value of $\psi$ on $5A$ is $0$.
\end{proof}

Our second proof uses less information about the group $G_\cL$, but requires a better understanding of the resolution of the ideal $I_\cL$.

\begin{proof}[Second proof]
We handle both configurations simultaneously.  Let $d$ be the number of lines in the configuration, and let $\Phi_d$ be the product of the lines in the configuration.  Therefore $d = 21$ if $\cL = \cK$ and $d =45$ if $\cL = \cW$.  Recall that $\Phi_d$ is the only invariant form of degree $d$ up to scalars.  We clearly have $\Phi_d\in I_{\cL}^{(2)}$.

We claim that, in order to establish the desired conclusion $\Phi_d\not\in I_{\cL}^{2}$, it is sufficient to show that the degree $d$ component $(I^{(2)}_{\cL}/I^{2}_{\cL})_d$ is a  one-dimensional trivial representation of $G$.  Indeed, suppose that  is spanned by a nonzero element $\bar f$. Pick a representative $f\in I^{(2)}_{\cL}\setminus I_{\cL}^{2}$ for $\bar{f}$. Since $g(\bar{f})= \bar{f}$ by the assumption that $G$ acts trivially, it follows that $g(f)-f\in I_{\cL}^{2}$ for any $g\in G$. Summing over the group elements yields $$\sum_{g\in G} g(f) - |G|\cdot f \in I_{\cL}^{2},$$ which shows that the $G$-invariant polynomial $\sum_{g\in G} g(f)$ is not in  $I_{\cL}^{2}$, since $f\not \in I_{\cL}^{2}$.  Then $\sum_{g\in G} g(f)$ is a nonzero multiple of $\Phi_d$ since $\Phi_d$ is the only invariant form of degree $d$ up to scalars, and we conclude that $\Phi_d\notin I_{\cL}^2$.

The rest of the proof will  aim to establish that $(I^{(2)}_{\cL}/I^{2}_{\cL})_d$ is a one dimensional vector space having trivial $G$ action. In order to do this, the key idea is to use the action of $G$ on a free resolution of $I^2_{\cL}$ in order to study the action of $G$ on the quotient  $I^{(2)}_{\cL}/I^{2}_{\cL}$.  Recall from Propositions \ref{prop-gensKlein} and \ref{prop-gensWiman} that the minimal free resolution of $I_{\cL}$ has the  form
$0\to M\to N\to I_{\cL}\to 0$, where $M=S^2$ and $N=S^3$. Since $ I_{\cL}$ is an almost complete intersection, the minimal free resolution for $I^2_{\cL}$ is given by the following complex (see e.g.\cite[Theorem 2.5]{NS16})
\begin{equation}\label{eq-res}
0\to \bigwedge^2M \to \bigwedge^1 M \otimes_S \Sym^1 N\to \Sym^2 N \to I^2_{\cL}\to 0.
\end{equation}
We record below the explicit form of this resolution for the two ideals of interest to us, with particular attention to the graded twists:
$$0\to S(-24)\to S^3(-21)\oplus S^3(-19)\to S^6(-16) \to I^2_{\cK}\to 0$$
$$0\to S(-48)\to S^3(-43)\oplus S^3(-37)\to S^6(-32) \to I^2_{\cW}\to 0.$$
Notice that the last free module in the resolution in both situations has rank one and is generated in degree $d+3$.
Since in our setting we have $H^0_m(S/I^2_{\cL})=I^{(2)}_{\cL}/I^2_{\cL}$, we can apply local duality to perform the following computations
$$\left(I^{(2)}_{\cL}/I^2_{\cL}\right)_d=H^0_m(S/I^2_{\cL})_d=\Ext^3_S(S,S/I^2_{\cL})^\vee_{-d-3}=\Ext^2_S(S,I^2_{\cL})^\vee_{-d-3}.$$

Thus to compute the vector space dimension of $\left(I^{(2)}_{\cL}/I^2_{\cL}\right)_d$ as well as the group action on this vector space it suffices to examine $\Ext^2_S(S,I^2_{\cL})_{-d-3}$.  Applying the functor $\Hom_S(-,S)$ to the resolutions displayed above and restricting to degree $-d-3$ gives in both cases that $\Ext^2_S(S,I^2_{\cL})_{-d-3}=\Hom_\C((\bigwedge^2M)_{d+3},\C)$ is a one dimensional vector space spanned by the dual of the generator of the last free $S$-module in the resolution (\ref{eq-res}). It remains to show that $G$ acts trivially on $(\bigwedge^2M)_{d+3}$. Let $\{e_1, e_2\}$ be a basis for the free module $M=S^2$. Then $(\bigwedge^2M)_{d+3}=\spn\{e_1\wedge e_2\}$ and it is in turn sufficient to show that $G$ acts trivially on $e_1$ and $e_2$ or equivalently on $M/\frakm M$.

Towards this goal, we begin by analyzing the group action on the minimal free resolution of $S/I_{\cL}$, which is given by $0\to M\to N\to S\to S/I_{\cL}\to 0$. Fix  an element $g\in G$. Denote by $S'$ the $S$-module that is isomorphic to $S$ as a ring, but carries a right $S$-module structure given by $f\cdot s=f\cdot g(s)$ for any $f\in S'$ and $s\in S$. Since $S'$ is a Cohen-Macaulay $S$-module and $S$ is regular we have that $S'$ is a flat $S$-module. Tensoring the resolution for $I_{\cL}$ with $S'$ gives an exact complex $0\to M\otimes_SS'\to N\otimes_SS' \to S'\to S'/I_{\cL}\to 0$. The two resolutions fit into the rows of the commutative diagram below, with vertical maps obtained by lifting the map $\phi:S \to  S'$ that maps $1\mapsto 1$, denoted by the equality symbol. Notice that this map sends $s=1\cdot s \in S \mapsto 1\cdot s= g(s)\in S'$, thus this map represents the action of $g$ on $S$.
\begin{equation*}\label{cd}
\begin{CD}
0     @>>>  M  @>J>> N @>\Delta>> R \\
@.  @VVCV        @VVBV  @|\\
0     @>>>  M \otimes_S S' @>J'>>  N \otimes_S S' @>\Delta'>>   S' \end{CD}
\end{equation*}
In the top row of the above diagram, $J$ denotes the  Hilbert-Burch matrix for $I_{\cL}$ and $\Delta$ denotes the vector of signed maximal minors of this matrix. By Propositions \ref{prop-gensKlein} and \ref{prop-gensWiman} the Hilbert-Burch matrix $J$ is the Jacobian matrix of the two smallest degree invariants of the relevant group acting on the set $\cL$. In the bottom row of the above diagram, $J'$ and $\Delta'$ are obtained by letting $g$ act on each of the entries of $J$ and $\Delta$ respectively.

Let $A_g$ be the matrix representing the action of $g$ on $S_1$. From Lemma \ref{lem-chainrule} we have that $J'=g(J)=A_g^{-1}J$. Next we seek an analogous description for $\Delta'$. Since $\Delta$ is the set of $2\times 2$ minors of $J$, we see that $\Delta^T=\wedge^2 J$.  Thus we have $g(\Delta)^T=g(\wedge^2 J)=\wedge^2(A_g^{-1}J)$.  We compute this by applying the $\wedge^2$ functor to the following commutative diagram as shown
\[
\begin{CD}
 M  @>J>> N \\
@|      @VVA_g^{-1}V \\
M  @>J'=A_g^{-1}J>>  N
\end{CD}
\qquad \Longrightarrow \qquad
\begin{CD}
 \wedge^2 M  @>\Delta^T =\wedge^2J>> \wedge^2 N \\
@|      @VV\wedge^2A_g^{-1}V \\
\wedge^2 M   @>(\Delta')^T>>  \wedge^2 N
\end{CD}
\]
It follows from the second diagram that
$$(\Delta')^T=(\wedge^2A_g^{-1})\Delta^T=\Cof(A_g^{-1})\Delta^T=\det(A_g^T)A_g^T\Delta^T=(\Delta A_g)^T,$$ were $\Cof(A_g^{-1})$ denotes the cofactor matrix and we use the property $\det(A_g)=1$  for all elements of $G$. Thus we conclude that $\Delta'=\Delta \cdot A_g$.

 Next we  proceed to determine the maps labeled $B$ and $C$ in our first diagram. The rightmost square gives $\Delta=\Delta'B$ or, equivalently, $\Delta=\Delta A_g B$. Hence we can pick the lifting $B=A_g^{-1}$. The leftmost square gives $BJ=J'C$, which becomes with our choice for $B$ the identity $A_g^{-1}J=A_g^{-1}JC$. Thus one can further pick $C=I_3$. Any other choices for $B$ and $C$ compatible with the above commutative diagram will be homotopic to the choices we made above. Since any pair of homotopic maps induce the same map on the quotient $M/\frakm M$, it follows that the action of $g$ on any basis elements of $M$ is the same as the action of $C$, namely the identity. Using the reductions made in the beginning of the proof, this finishes the argument. \end{proof}

One final result that we will need to compute the resurgence is a computation of the regularity of the ordinary powers of the ideal $I_{\cL}$.

\begin{proposition}\label{prop-regularity}
If $r\geq 2$, then $\reg(I_\cK^r) = 8r+6$ and $\reg(I_{\cW}^r) = 16r+14$.
\end{proposition}
\begin{proof}
The ideal $I_\cL$ defines a reduced collection of points in $\P^2$ and it is generated by $3$ homogeneous polynomials of the same degree $d$, with $d=8$ for $\cL = \cK$ and $d = 16$ for $\cL = \cW$ (see Propositions \ref{prop-gensKlein} and \ref{prop-gensWiman}).  These properties allow us to use \cite[Theorem 2.5]{NS16} to explicitly compute the minimal free resolution of any power $I^r_\cL$.  From the minimal free resolution we determine that $\reg(I^r_\cL) = rd+d-2$.
\end{proof}

We can now give the proof of Theorem \ref{thm-resIntro}, computing the resurgence of the ideal of the Klein and Wiman configurations of points.

\begin{proof}[Proof of Theorem \ref{thm-resIntro}]
By Proposition \ref{prop-containmentFailure} and the Ein-Lazarsfeld-Smith theorem \cite{ELS01}, we need to show that if $m,r$ are positive integers with $\frac{3}{2}< \frac{m}{r} $ then $I_{\cL}^{(m)}\subset I_{\cL}^r$; let $m,r$ be such integers.  Recall that if $\alpha(I_{\cL}^{(m)}) \geq \reg I_{\cL}^r$ then the containment $I_{\cL}^{(m)}\subset I_{\cL}^r$ holds by \cite[\S 2.1]{BH10}.

In the case of the Klein ideal $I_{\cK}$, we estimate $\alpha(I_{\cK}^{(m)})\geq m \widehat\alpha(I_{\cK}) \geq \frac{58}{9}m$ by Corollary \ref{cor-KleinLowerThy}.  Since $\reg(I_{\cK}^r) = 8r+6$ by Proposition \ref{prop-regularity}, we see that the containment $I_{\cK}^{(m)}\subset I_{\cK}^r$ holds whenever $$\frac{58}{9}m \geq 8r+6.$$  It is easy to see that this inequality holds for any positive integers $m,r$ with $\frac{3}{2}< \frac{m}{r}$.

For the Wiman ideal $I_{\cW}$, we use Corollary \ref{cor-WimanWaldschmidt} to estimate $\alpha(I_{\cW}^{(m)})\geq \frac{27}{2}m$.  From $\reg(I_{\cW}^r) = 16r+14$, we conclude that the containment $I_{\cW}^{(m)}\subset I_{\cW}^r$ holds if $$\frac{27}{2}m\geq 16r+14.$$ Again, the inequality holds for any positive integers $m,r$ with $\frac{3}{2}< \frac{m}{r}$.
\end{proof}

\begin{remark}
Note that in the case of the Klein configuration we only needed to use the weaker lower bound on $\widehat\alpha(I_{\cK})$ coming from Corollary \ref{cor-KleinLowerThy}.
\end{remark}

\section{Positive characteristic}\label{sec-posChar}

The Klein configuration can be defined over fields of characteristics other than 0; to be able to define the coordinates of the points of the configuration one needs the base field to contain a root of $x^2+x+2=0$ (see section 1.4 of \cite{BNAL}) and the field needs to be sufficiently large that the resulting 49 points are different.
There is reason to believe that it behaves much as it does over the complex numbers
except for characteristic 7 (see \cite{S14}). The fact that characteristic 7 is special
is suggested by the fact that it is the only characteristic for which
$x^2+x+2=0$ has a double root (in this case $x=3$).
We now consider the case of characteristic 7, as given in \cite{GR90}.

The configuration  is described
geometrically in \cite{GR90} in a very simple way.
Consider the conic $C$ defined by $x^2+y^2+z^2=0$. Over the finite field $K={\mathbb F}_7$ of characteristic 7,
$C$ has 8 points and thus 8 tangents. There are 21 $K$-lines in $\P^2_K$ that do not intersect $C$ in
a $K$-point; these are the 21 lines of the Klein configuration.
There are also 21 $K$-points of $\P^2$ not on any of the 8 $K$-tangents to $C$;
these are the 21 quadruple singular points of the Klein configuration.
The remaining 28 singular points, which are triple points, are the $K$-points
on a tangent but not on $C$.

\begin{theorem}
Let $I$ be the ideal of the 49 Klein points over $K={\mathbb F}_7$.
Then $\widehat{\alpha}(I)=6.25$ and
$1.28\leq\widehat{\rho}(I)\leq1.44<\rho(I)=3/2$.
\end{theorem}

\begin{proof}
To verify $\widehat{\alpha}(I)=6.25$, note that the $28 = {8 \choose 2}$ triple points
are the pairwise intersections of the 8 tangent lines. Thus
they comprise a star configuration on these 8 lines, for which $\alpha(I^{(2)})$
is known to be the degree of the product $G$ of the forms defining the 8 lines \cite{BH10}.
Let $F$ be the product of the linear forms for the 21 Klein lines. Then $F^2G$ vanishes on each of the 49 points
with order 8, so $F^2G\in I^{(8)}$, hence $\alpha(I^{(8)})\leq \deg(F^2G)=50$, so
$\widehat{\alpha}(I)\leq 50/8=6.25$.
(We note that
this argument does not apply to the Klein configuration
of 49 points in characteristic 0, since the 28 points are not in that case a star configuration.
Alternatively, $\alpha(I^{(8)})=50$
can be checked in characteristic 7 explicitly using Macaulay2. In contrast, in characteristic 0
Macaulay2 gives $\alpha(I^{(8)})=54$.)

For the lower bound it is enough to show that
$\alpha(I^{(m)})\geq 6.25m=\frac{50m}{8}$ for infinitely many $m\geq 1$.
We used the general methods of \cite{CHT11} to discover the argument
we now give. We will show that $\alpha(I^{(8m)})\geq 50m$ for all $m\geq 1$.

Any form $H$ of degree $d\leq 50m$ vanishing to order at least $8m$ at the 49 Klein points
is divisible by $FG$. This is because $FG$ is a product of $21+8=29$ linear factors, and each factor
vanishes on either 7 or 8 of the 49 points. But $7(8m) > 50m$, so by B\'ezout's Theorem,
each linear factor of $FG$ is a factor of $H$. Factoring these out leaves a form $H'$
of degree $50m-29$ vanishing to order at least $8m-4$ at the 21 quadruple points
and to order at least $8m-5$ at the 28 triple points.
Since each linear factor of $F$ vanishes at 4 of the quadruple points and 4 of the triple points
and since $50m-29 < 4(8m-4)+4(8m-5)$ as long as $m\geq 1$, it follows, again by B\'ezout, that $F$ divides $H'$, and so for $m\geq 1$ it follows that $F^2G$ divides $H$.

Dividing $H$ by $F^2G$ gives a form $H^*$ of degree $d-50\leq 50(m-1)$
vanishing to order at least $8(m-1)$ at the 49 Klein points. Up to scalars, it follows by induction
that $H=(F^2G)^m$ and thus that $d=50m$.

Since Macaulay2 gives $\alpha(I)=8$ and $\omega(I)=9$, applying \eqref{GHvTbounds} gives
the bounds $1.28=\alpha(I)/\widehat{\alpha}(I)\leq\widehat{\rho}(I)\leq \omega(I)/\widehat{\alpha}(I)=9/6.25=1.44$.

Finally we show that $\rho(I) = 3/2$.  Macaulay2 demonstrates the failure of containment $I^{(2)}\not\subset I^3$.  Suppose $\frac{m}{r} > \frac{3}{2}$; we need to check the containment $I^{(m)}\subset I^r$ holds.  First, if  $r\leq 7$ then we can check with Macaulay2 that $I^{(m)} \subset I^r$; it suffices to only consider $m = \lceil 3r/2\rceil.$ So suppose $r\geq 8$.
By \cite{BH10}, if $\alpha(I^{(m)})\geq \reg(I^r)$ then the containment $I^{(m)}\subset I^r$ holds.  Now we estimate $\alpha(I^{(m)}) \geq 6.25m$ and $$\reg(I^r) \leq 2\reg(I) + (r-2)\omega(I)$$ by \cite[Theorem 0.5]{C2-07}.  Macaulay2 gives $\reg(I) = 12$, so this simplifies to $$\reg(I^r) \leq 9r+6.$$  Since $\frac{m}{r} > \frac{3}{2}$ we have $m \geq \frac{3}{2}r + \frac{1}{2}$, and since $r \geq 8$ we have $$\alpha(I^{(m)}) \geq \frac{25}{4}m \geq \frac{75}{8}r+\frac{25}{8} \geq 9r+6 \geq \reg(I^r).$$  Therefore the containment $I^{(m)} \subset I^r$ holds and we conclude $\rho(I)=3/2$.
\end{proof}

\section*{Appendix: Macaulay2 scripts}

Here we give a sample of some of the Macaulay2 scripts used to study the Klein configuration.    Very similar scripts are also useful for studying the Wiman configuration.

We begin by defining the field $\Q(\zeta)$ and inputting the group $G$ of order $168$ in terms of its generators.  We form words in the generators until all $168$ elements have been constructed.

\begin{Verbatim}[fontsize=\footnotesize]
--define ground field and ring; pick one of the fields
K = toField(QQ[w]/(w^6+w^5+w^4+w^3+w^2+w+1));
--K = ZZ/4733; w = 7;
S = K[x,y,z];

--some constants
a = w - w^6; b = w^2 - w^5; c = w^4 - w^3;
h = (w + w^2 + w^4 - w^6 - w^5 - w^3)/7;

--generators of the group G of order 168
r0 = map(S,S,{x,y,z});
r1 = map(S,S,{w^4*x,w^2*y,w*z});
r2 = map(S,S,{z,x,y});
r3 = map(S,S,{h*(a*x+b*y+c*z),h*(b*x+c*y+a*z),h*(c*x+a*y+b*z)});

--compute the elements of G
groupGens = set {r0,r1,r2,r3};
G = set {r0};
for i from 1 to 8 do (G = toList (G ** groupGens); G = set apply(G,r->r_0*r_1));
G = toList G;
#G --verify 168
\end{Verbatim}
We next define a handful of scripts that are useful to construct invariants.  We define the Reynold's operator for a group $G$ acting on the homogeneous coordinate ring $S$.  We also introduce several differential determinants used to compute invariants.
 \begin{Verbatim}[fontsize=\footnotesize]
--reynolds(G,f) computes the Reynold's operator R_G(f)
reynolds = method();
reynolds(List,RingElement) := (G,f) -> (card=#G; sum apply(G, g-> g(f))/card)

--define Jacobian, Hessian, and bordered Hessian determinants
jacobianDet = method();
jacobianDet(RingElement,RingElement,RingElement) := (f,g,h) -> (
    determinant (jacobian matrix{{f,g,h}}))

hessianDet = method();
hessianDet(RingElement) := (f) -> (
    determinant jacobian transpose jacobian matrix{{f}});

borderedHessianDet = method();
borderedHessianDet(RingElement,RingElement) := (f,g) -> (
    M1 = jacobian transpose jacobian matrix{{f}};
    M2 = jacobian matrix{{g}};
    M3 = M1 | M2;
    M4 = M2 || matrix {{0}};
    M5 = M3 || transpose(M4);
    determinant M5
    )
\end{Verbatim}
We apply these methods to define the basic invariants $\Phi_4,\Phi_6,\Phi_{14},\Phi_{21}$.
\begin{Verbatim}[fontsize=\footnotesize]
--define invariants
f4 = 3 * reynolds(G,x^3*y);
f6 = -1/54*hessianDet(f4);
f14 = 1/9*borderedHessianDet(f4,f6);
f21 = 1/14*jacobianDet(f4,f6,f14);
\end{Verbatim}
The relation between $\Phi_{21}^2$ and the other invariants can then be verified immediately.  We next use the Reynold's operator to find a line in the Klein configuration $\cK$.  We then compute individual triple and quadruple points, as well as their orbits and the ideal $I_\cK$ of all the points in the configuration.

\begin{Verbatim}[fontsize=\footnotesize]
--compute lines and points in the configuration
H = {r0,r3};
line1 = 14*reynolds(H,x);
line2 = r1(line1);
line3 = r2(line1);
quadPoint = trim ideal(line1,line2);
tripPoint = trim ideal(line1,line3);
quadPoints = intersect unique apply(G,r->trim r(quadPoint));
tripPoints = intersect unique apply(G,r->trim r(tripPoint));
KleinPoints = intersect(quadPoints,tripPoints);
degree quadPoints -- verify 21
degree tripPoints -- verify 28
degree KleinPoints -- verify 49
\end{Verbatim}

We now introduce a script that computes the linear series $T_d(-m_4E_4 - m_3E_3)$.  It is very useful for experimenting with the Klein configuration.  Especially for gathering evidence, it is frequently useful to first work in characteristic $p$ when doing substantial computations of this type.  To compute $T_d(-m_4E_4-m_3E_3)$ we compute the kernel of the linear map $$T_d \to \OO_{p_4}/\frakm_{p_4}^{m_4} \oplus \OO_{p_3}/\frakm_{p_3}^{m_3}.$$ We write down the matrix of this map in terms of natural bases.  Computation of the matrix is sped up by first determining the images of pure monomials and then multiplying the results to get the images of arbitrary monomials.

\begin{Verbatim}[fontsize=\footnotesize]
--ring for weighted projective space and quotient maps
T = K[v1,v2,v3,Degrees=>{4,6,14}];
phi = map(S,T,{f4,f6,f14});

--series({d,-m4,-m3}) gives a basis for T_d(-m4*E_4 - m3*E_3)
series = method();
series(List) := C -> (
    d = C_0;  m4 = -C_1;  m3 = -C_2;
    --Taylor series maps
    use S;  I = ideal(x,y);
    S3 = S/(I^(m3));
    S4 = S/(I^(m4));
    --psiTrip recenters triple point to (0,0)
    use S;  psiTrip = map(S3,S,{x+z,y+z,z});
    tripMap = psiTrip*phi;
    --psiQuad recenters quadruple point to (0,0)
    use S;  psiQuad = map(S4,S,{x-(-w^4-1)*z,y-(w^5+w^3+w)*z,z});
    quadMap = psiQuad*phi;

    --need to compute tripMap,quadMap applied to monomial basis of T in degree d.
    --speed this up by computing maps on "pure" monomials first
    v1Trip_0= tripMap(1_T);  v2Trip_0= tripMap(1_T);  v3Trip_0= tripMap(1_T);
    v1Trip_1= tripMap(v1);   v2Trip_1= tripMap(v2);   v3Trip_1= tripMap(v3);
    for i from 2 to floor(d/4)  do v1Trip_i = v1Trip_(i-1) * v1Trip_1;
    for i from 2 to floor(d/6)  do v2Trip_i = v2Trip_(i-1) * v2Trip_1;
    for i from 2 to floor(d/14) do v3Trip_i = v3Trip_(i-1) * v3Trip_1;
    v1Quad_0= quadMap(1_T);  v2Quad_0= quadMap(1_T);  v3Quad_0= quadMap(1_T);
    v1Quad_1= quadMap(v1);   v2Quad_1= quadMap(v2);   v3Quad_1= quadMap(v3);
    for i from 2 to floor(d/4)  do v1Quad_i = v1Quad_(i-1) * v1Quad_1;
    for i from 2 to floor(d/6)  do v2Quad_i = v2Quad_(i-1) * v2Quad_1;
    for i from 2 to floor(d/14) do v3Quad_i = v3Quad_(i-1) * v3Quad_1;

    --compute tripMap, quadMap on arbitrary monomials using result for pure monomials
    tripMonomialImage = method();
    tripMonomialImage(RingElement) := monom -> (
        e1 = degree(v1,monom);  e2 = degree(v2,monom);  e3 = degree(v3,monom);
        v1Trip_(e1)*v2Trip_(e2)*v3Trip_(e3)
        );
    quadMonomialImage = method();
    quadMonomialImage(RingElement) := monom -> (
        e1 = degree(v1,monom);  e2 = degree(v2,monom);  e3 = degree(v3,monom);
        v1Quad_(e1)*v2Quad_(e2)*v3Quad_(e3)
        );

    --build system of linear equations over K
    monomialList = flatten entries basis(d,T);
    M1 = lift(matrix{apply(monomialList,tripMonomialImage)},S);
    M2 = lift(matrix{apply(monomialList,quadMonomialImage)},S);
    M1c = lift((coefficients(M1))_1,K);
    M2c = lift((coefficients(M2))_1,K);
    M3c = M1c || M2c;

    --compute kernel of matrix and convert vectors to elements of T
    flatten entries (basis(d,T)*(generators kernel(M3c)))
    )
\end{Verbatim}

For instance, the following commands compute the curves of class $42H-8E_3$ and $144H - 4E_4-27E_3$ in $<1$ second and $\approx 90$ seconds in characteristic 0, respectively.

\begin{Verbatim}[fontsize=\footnotesize]
timing series({42,0,-8})
timing series({144,-4,-27})
\end{Verbatim}

Finally, we give a script to verify Theorem 5.7.  The script finds the invariant curves of negative self-intersection which meet all lower degree invariant curves of negative self-intersection nonnegatively.  The list \texttt{negCurveList} of negative curves begins with just the class $21H-4E_4-3E_3$ of the line configuration, and additional classes are added as they are found.

\begin{Verbatim}[fontsize=\footnotesize]
--negative curve search
dMax = 200;  negCurveList = {{21,-4,-3}};

--intersection numbers
intersection = (C,D) -> C_0*D_0 - 21*C_1*D_1-28*C_2*D_2;
selfIntersection = (C) -> intersection(C,C);

--posOnNegCurves(C) returns true if C.D >= 0 for all curves D in negCurveList
posOnNegCurves = (C) -> (
    numCurves = #negCurveList;
    for i from 0 to numCurves-1 do (
        if (intersection(C,negCurveList_i) < 0) then return false;
        );
    return true;
    )

--find effective invariant curves of negative self-intersection meeting previous
--such curves nonnegatively
for e from 2 to floor(dMax/2) do (
    d = 2*e;
    for m3 from 0 to floor(d/4) do (
        m4 = 0;
        while(selfIntersection({d,-m4,-m3}) >= 0) do m4 = m4+1;
        C = {d,-m4,-m3};
        if posOnNegCurves(C) and (m3==0 or selfIntersection({d,-m4,-m3+1})>=0) then (
            << "(" << d << "," << m4 << "," << m3 << ") " << timing (L = series(C););
            if #L > 0 then (
                << "    Negative Curve!";
                negCurveList = append(negCurveList,C);
                );
            << endl << flush;
            )
        )
    )
\end{Verbatim}




\end{document}